\documentclass{amsart}
\usepackage{amssymb, amsmath}
\usepackage{enumerate}

\usepackage[fleqn,tbtags]{mathtools}
\mathtoolsset{showonlyrefs,showmanualtags}

\usepackage{mathrsfs}
\usepackage{amscd}
\usepackage[active]{srcltx}
\usepackage{verbatim}
\usepackage{color}

\usepackage{hyperref}

 \newtheorem{thm}{Theorem}[section]
 \newtheorem{cor}[thm]{Corollary}
 \newtheorem{lem}[thm]{Lemma}
 \newtheorem{prop}[thm]{Proposition}
 \theoremstyle{definition}
 \newtheorem{defn}[thm]{Definition}
 \theoremstyle{remark}
 \newtheorem{rem}[thm]{Remark}
 \newtheorem*{ex}{Example}
 \numberwithin{equation}{section}

\def\N{{\mathbb N}}

\def\R{{\mathbb R}}
\def\C{{\mathbb C}}

\newcommand{\dD}{\mathbb{D}}
\newcommand{\E}{{\mathbb E}}
\renewcommand{\P}{{\mathbb P}}
\newcommand{\A}{{\mathscr A}}
\newcommand{\F}{{\mathscr F}}
\newcommand{\G}{{\mathscr G}}
\renewcommand{\H}{{\mathscr H}}

\newcommand{\calE}{{\mathscr E}}

\newcommand{\cG}{\mathcal{G}}

\newcommand{\cS}{{S}}
\newcommand{\cE}{\mathcal{E}}

\newcommand{\cJ}{\mathscr{J}}
\newcommand{\cI}{{I}}
\newcommand{\cD}{{D}}

\newcommand{\g}{\gamma}
\renewcommand{\d}{\delta}
\newcommand{\eps}{\varepsilon}
\newcommand{\om}{\omega}
\renewcommand{\O}{\Omega}

\newcommand{\Dom}{{\mathsf D}}

\renewcommand{\Re}{{\rm Re}\,}

\newcommand{\calL}{{\mathscr L}}
\newcommand{\n}{\Vert}
\newcommand{\one}{{{\bf 1}}}
\newcommand{\embed}{\hookrightarrow}
\newcommand{\s}{^*}
\newcommand{\lb}{\langle}
\newcommand{\rb}{\rangle}

\newcommand{\EE}{{\mathbb E}}
\newcommand{\CE}[2]{\EE(#1|#2)}
\newcommand{\Om}{\Omega}
\newcommand{\cprime}{$'$}
\newcommand{\wt}{\widetilde}

\allowdisplaybreaks
\begin{document}

\title
[Stochastic integration in Banach spaces]{Stochastic integration in Banach spaces -- \\ a survey}

\author{Jan van Neerven}
\address{Delft Institute of Applied Mathematics\\
Delft University of Technology \\ P.O. Box 5031\\ 2600 GA Delft\\The
Netherlands} \email{J.M.A.M.vanNeerven@tudelft.nl}

\author{Mark Veraar}
\address{Delft Institute of Applied Mathematics\\
Delft University of Technology \\ P.O. Box 5031\\ 2600 GA Delft\\The
Netherlands} \email{M.C.Veraar@tudelft.nl}

\author{Lutz Weis}
\address{Department of Mathematics\\
Karlsruhe Institute of Technology (KIT)\\
D-76128  Karls\-ruhe\\Germany}
\email{Lutz.Weis@kit.edu}

\begin{abstract}
This paper presents a brief survey of the theory of stochastic integration in Banach spaces.
Expositions of the stochastic integrals in martingale type $2$ spaces and UMD spaces are presented,
as well as some applications of the latter to vector-valued Malliavin calculus and the stochastic
maximal regularity problem. A new proof of the stochastic maximal regularity theorem
is included.
\end{abstract}



\date\today

\thanks{The first named author is supported by VICI subsidy 639.033.604
of the Netherlands Organisation for Scientific Research (NWO). The second author
is supported by VENI subsidy
639.031.930
of the Netherlands Organisation for Scientific Research (NWO).
The third named author is supported by a grant from the
Deutsche Forschungsgemeinschaft (We 2847/1-2).}
\thanks{This paper is based on the lectures delivered by the first-named author during the Semester on Stochastic
Analysis and Applications at the Bernoulli Centre, EPF Lausanne. JvN wishes to thank the organisers Robert Dalang,
Marco Dozzi, Franco Flandoli, and Francesco Russo for the excellent conditions and the enjoyable atmosphere.}

\maketitle

\section{Introduction}
Stochastic calculus was developed in the 1950s in the fundamental work of It\^o.
In its simplest form, the construction of
the It\^o stochastic integral with respect to a Brownian motion $(B_t)_{t\ge 0}$
relies on an $L^2$-isometry,
which asserts that if $\phi: \R_+\times \Omega \to \R$ is an adapted simple process, then
$$ \E \Big| \int_0^\infty \phi_t \,dB_t \Big|^2 = \E \int_0^\infty |\phi_t|^2\,dt.$$
This isometry is used to extend the stochastic integral to arbitrary progressively measurable
processes satisfying $\E \int_0^\infty |\phi_t|^2\,dt < \infty.$
The stochastic integral process $t\mapsto \int_0^t \phi_s\,dB_s$
defines a continuous $L^2$-martingale, and by means of stopping time techniques
the integral can be extended to all progressively measurable
processes satisfying $$\int_0^\infty |\phi_t|^2\,dt < \infty \ \ \hbox{almost surely}.$$

It was immediately realised that the above programme generalises {\em mutatis mutandis} to stochastic integrals
of progressively measurable processes with values in a Hilbert space $H$.
Some of the early works in this direction include \cite{BD, CM, Cur, Dal, Met1}.
More generally, if $H'$ is another Hilbert space one may allow operator-valued integrands
with values in the space of Hilbert-Schmidt operators $\calL_2(H,H')$ to define an $H'$-valued stochastic
integral with respect to an $H$-cylindrical Brownian motion. This integral was
popularised by Da Prato and Zabczyk, who used it to study stochastic partial differential equations (SPDE)
by functional analytic and operator theoretic methods \cite{DPZ, DPZ2}; see also \cite{Met2}.

From the point of view of SPDE the limitation to the Hilbert space framework is rather restricting, and various authors
have attempted to extend the theory of stochastic integration to more general classes of Banach spaces.
It was realised soon that a stochastic integral for square integrable functions with values in a
Banach space $X$ can be defined
if $X$ has type $2$ \cite{HJP}, whereas a bounded measurable function $f:[0,1]\to \ell^p$ may fail to be stochastically
integrable for $1\le p<2$ \cite{Yo}; see \cite{RS} for more detailed results and examples along these lines.
A systematic theory
of stochastic integration in $2$-smooth Banach spaces was developed by Neidhardt in his 1978 PhD thesis and,
independently, by Belopol{\cprime}skaya and Dalecky, \cite{BelDal} and Dettweiler \cite{Dett91}
independently developed a parallel theory for martingale type $2$ spaces.
Interestingly, Pisier \cite{Pi75} had already shown in 1975
that a Banach space has an equivalent $2$-smooth norm
if and only if it has martingale type $2$.
The stochastic integrals of Neidhardt and Dettweiler were further developed and
applied to SPDEs by Brze\'zniak
\cite{Brz1, Brz2, Brz3, BrzMil}. We shall briefly
summarize the martingale type $2$ approach in Section
\ref{sec:Mtype2}.

Along a different line, the fundamental work of Burkholder \cite{Bu1, Bu2} showed
that many of the deeper inequalities in the theory of martingales
extend to a class of Banach spaces in which martingale differences are unconditional, nowadays called the
class of UMD Banach spaces. These spaces were characterised by Burkholder \cite{Bu1} and Bourgain \cite{Bou83} as
precisely those Banach spaces $X$ for which the Hilbert transform on $L^p(\R)$ extends boundedly to $L^p(\R;X)$.
As a consequence, UMD spaces provide a natural framework for vector-valued harmonic analysis, and indeed
large parts of the theory of singular integrals have by now been extended to UMD spaces
\cite{Bour86, GirWe03, HHN02,Hyt06, McCon84, StrWeis08, Zim89}.

The probabilistic definition of the UMD property in terms of martingale differences suggests the possibility
to develop stochastic calculus in UMD spaces. The first result in this direction is due to Garling \cite{Gar},
 who proved a two-sided $L^p$-estimate for the stochastic
integral of an adapted simple process $\phi$ with values in a UMD space in
terms of the stochastic integral of $\phi$ with respect to an independent
Brownian motion.  McConnell \cite{McCon89} proved decoupling inequalities for tangent martingale difference sequences
and used them to obtain a sufficient condition for stochastic integrability of an UMD-valued process with
respect to a Brownian motion in terms of the almost sure stochastic integrability of its trajectories with respect
to an independent Brownian motion.
The ideas of Garling and McConnell have been streamlined and extended in a systematic way
by the present authors \cite{NVWco, NVW1, NW1} and applied to SPDEs \cite{BNVW,NVW3, NVW12eq, NVW10}.
A key idea in obtaining two-sided estimates of Burkholder-Gundy type is to measure the integrand
in a norm that is custom-made for the Gaussian setting, rather than in the traditional Lebesgue-Bochner norms. In
an operator-theoretic language, these Gaussian norms are given in terms of certain $\gamma$-radonifying operators
(see Section \ref{sec:radonifying} for the relevant definitions).
The main aim of this paper is to provide a coherent presentation of this theory and some of its applications,
in particular to the vector-valued Malliavin calculus and the stochastic maximal $L^p$-regularity problem.
In the final section of this paper we discuss some recent $L^p$-bounds for vector-valued Poisson stochastic integrals.

Let us mention that various different approaches to stochastic integration in Banach spaces exist in the
literature, e.g., \cite{BroDin, GirRus, GirRus2, MetPel}.

We finish the introduction by fixing some notation. All vector spaces are real.
Throughout the paper, $H$ and $\H$ are fixed Hilbert spaces.
We will always identify Hilbert spaces with their duals via the Riesz representation theorem.
All random variables are supposed to be defined on a fixed probability space $(\O,\P)$.

\section{Isonormal processes}

It is a well-known result in the theory of Gaussian measures that an infinite-dimensional
Hilbert space $\H$ does not support a standard Gaussian measure (cf. \cite{Bog}). By this we mean that there exists
no Radon probability measure $\gamma$ on $\H$ with the property that for all $h\in \H$ of norm one
the image measure of $\gamma$ under the mapping $h:\H\to \R$ is standard Gaussian.
The following definition serves as a substitute.

\begin{defn}\label{def:isonormal}
An {\em $\H$-isonormal process} is a bounded linear mapping
$W: \H \to L^2(\O)$
with the following properties:
\begin{enumerate}
\item[\rm(i)] for all $h\in \H$ the random
variable $Wh$ is Gaussian;
\item[\rm(ii)] for all $h_1,h_2\in \H$ we have $\E (Wh_1 \cdot Wh_2) = [h_1,h_2].$
\end{enumerate}
\end{defn}
It is an easy exercise to check that for any Hilbert space $\H$, an $\H$-isonormal process does indeed exist.
The random variables $Wh$, $h\in H$, are jointly Gaussian, as every
linear combination $\sum_{j=1}^k c_j\, Wh_j = W(\sum_{j=1}^k c_j h_j)$ is Gaussian.
In particular this implies that if $h_1,\dots,h_k$ are orthogonal, then  $Wh_1, \dots, Wh_k$ are
independent. For more details we refer to \cite[Chapter 1]{Nu}.

\begin{ex} If $(B_t)_{t\ge 0}$ is a standard Brownian motion in $\R^d$, then the It\^o stochastic integral
 $$ W(f) := \int_0^\infty f(t)\,dB_t, \qquad f\in L^2(\R_+;\R^d),$$
defines an $L^2(\R_+;\R^d)$-isonormal process $W$. In the converse direction, if $W$ is an  $L^2(\R_+;\R^d)$-isonormal process,
we let $(e_j)_{j=1}^d$ denote the standard unit basis of $\R^d$ and
note that
$$ B_t^{(j)} := W(\one_{[0,t]}\otimes e_j), \qquad t\ge 0,$$
defines a standard Brownian motion for each $1\le j\le d$; these Brownian motions are independent
and define the coordinates of a standard Brownian motion in $\R^d$.
\end{ex}

\begin{defn}
An {\em $H$-cylindrical Brownian motion} is an $L^2(\R_+;H)$-isonormal process.
\end{defn}

\begin{defn} A {\em space-time white noise} on a domain $D\subseteq \R^d$ is
an $L^2(\R_+\times D)$-isonormal process.
\end{defn}
Under the natural identification $L^2(\R_+\times D) = L^2(\R_+; L^2(D))$, a space-time white noise
may be identified with an $L^2(D)$-cylindrical Brownian motion.

\section{Radonifying operators}\label{sec:radonifying}

Let $H\otimes X$ denote the linear space of all finite rank operators from $H$ to $X$.
Every element in $H \otimes X$ can be represented in the form
$\sum_{n=1}^N h_n\otimes x_n$, where $h_n\otimes x_n$ is the rank one operator mapping the vector
$h\in H$ into $[h,h_n]x_n \in X$. By a Gram-Schmidt orthogonalisation argument
we may assume that the vectors  $h_1,\dots,h_N$ are orthonormal in $H$.

Let $(\gamma_n)_{n\ge 1}$ be a Gaussian sequence, i.e., a sequence of independent real-valued standard Gaussian
random variables.

\begin{defn}\label{def:g-rad}
The Banach space $\g( H ,X)$ is defined as the completion of $ H \otimes X$ with
respect to the norm
$$ \Big(\Big\n \sum_{n=1}^N h_n\otimes x_n\Big\n_{\g(H,X)}^2\Big)^{1/2}
:=  \Big(\E \Big\n \sum_{n=1}^N \g_n x_n\Big\n^2\Big)^{1/2},
$$
where it is assumed that $h_1,\dots,h_N$ are orthonormal in $ H $.
\end{defn}
The quantity on the right-hand side is independent of the above representation as long as the
vectors in $H$ are taken to be orthonormal; this is an easy consequence
of the fact that the distribution of a Gaussian vector in $\R^N$ is invariant under orthogonal transformations.
As a result, the norm $\n\cdot\n_{\gamma(H,X)}$ is well defined.

The celebrated Kahane-Khintchine inequality asserts that for all $0<p,q<\infty$ there exists
a constant $\kappa_{q,p}\ge 0$, depending only on $p$ and $q$, such that
\begin{equation}\label{eq:KKradonifying}
\Big(\E \Big\n \sum_{n=1}^N \g_n x_n\Big\n^q \Big)^{1/q}
\le \kappa_{q,p}\Big( \E \Big\n \sum_{n=1}^N \g_n x_n\Big\n^p\Big)^{1/p}.
\end{equation}
Proofs can be found in \cite{DJT, KwWo89, LeTa}.
It was shown in \cite{LatOle99}
that the optimal constant is given by $\kappa_{q,p} = \max\{\frac{\|\g_1\|_q}{\|\g_1\|_p},1\}$.
In particular, $\kappa_{q,p}\leq C_p \sqrt{q}$ for $q\geq 1$.

It follows from \eqref{eq:KKradonifying} that for each $p\in [1,\infty)$ we obtain an equivalent
norm on $\gamma(H,X)$ if we replace
the exponent $2$ by $p$
in Definition \ref{def:g-rad}. The resulting space will be indicated by $\gamma^p(H,X)$.

The identity mapping on $H\otimes X$ extends to an injective and contractive embedding of
$\gamma(H,X)$ into $\calL(H,X)$, the space of all bounded linear operators from $H$ into $X$
(for the simple proof see \cite[Section 3]{NeeCMA}). We may thus identify $\gamma(H,X)$ with a
linear subspace in $\calL(H,X)$. Assuming this identification,
we call a bounded operator $T\in \calL(H,X)$ {\em $\gamma$-radonifying} if it belongs to
$\gamma(H,X)$.

\begin{ex}
 If $X$ is a Hilbert space, then we have an isometric isomorphism $$\gamma(H,X)  =\calL_2(H,X),$$
where $\calL_2(H,X)$ is the space of all Hilbert-Schmidt operators from $H$ to $X$.
\end{ex}

\begin{ex}
For $1\le p<\infty$ we have an isometric isomorphism of Banach spaces
\begin{equation}\label{eq:gammaLpX}
\gamma^p(H, L^p(\mu;X)) \simeq L^p(\mu;\gamma^p(H;X))
\end{equation} which is
obtained by associating with $f\in L^p(\mu;\gamma(H;X))$ the mapping
$h'\mapsto f(\cdot)h'$ from $H$ to $L^p(\mu;X)$.
The proof is an easy application of Fubini's theorem.
In particular, upon identifying $\gamma^p(H,\R)$ isomorphically with $H$, we obtain an
isomorphism of Banach spaces
\begin{equation}\label{eq:gammaLp}
\gamma^p(H, L^p(\mu)) \simeq L^p(\mu;H).
\end{equation}
\end{ex}

\section{Stochastic integration in martingale type $2$ spaces}\label{sec:Mtype2}

In this section we shall give a brief account of the construction of the It\^o stochastic integral
in martingale type $2$ spaces.
In order to bring out the analogy with the UMD approach more clearly we will first consider the
simpler case of deterministic integrands, for which it suffices to assume that
$X$ has type $2$.

\subsection{Deterministic integrands}

Let $(r_n)_{n\ge 1}$ be a Rademacher sequence, i.e., a sequence of independent random variables
taking the values $\pm 1$ with probability $\frac12$.

\begin{defn} Let $p\in [1,2]$.
 A Banach space $X$ has {\em type $p$} if there exists a constant $\tau\ge 0$ such that for all
finite sequences $(x_n)_{n=1}^N$ in $X$ we have
\begin{equation}\label{eq:type}
\E\Big\| \sum_{n=1}^N r_n x_n \Big\|^p \le \tau^p\sum_{n=1}^N  \| x_n\|^p.
\end{equation}
\end{defn}
The least admissible constant
is denoted by $\tau_{p,X}$. In the next subsection we will give some
examples of spaces with type $p$; in fact these examples have the stronger property of martingale type $p$.

In the proof of the next proposition we shall use the following randomisation identity.
If $(\xi_n)_{n\ge 1}$ is a sequence of independent
symmetric random variables in $L^p(\Omega;X)$, and if $(\wt r_n)_{n\ge 1}$ is an independent Rademacher sequence
defined on another probability space $(\widetilde \Omega,\widetilde \P)$,
then for all $N\ge 1$ we have
\begin{equation}\label{eq:typerandom}
\E\Big\| \sum_{n=1}^N \xi_n\Big\|^p =
\E\widetilde \E\Big\| \sum_{n=1}^N \widetilde r_n \xi_n\Big\|^p.
\end{equation}
This follows readily from Fubini's theorem, noting that for each $\wt\omega\in \wt\Omega$
the sequences $(\xi_n)_{n\ge 1}$ and $(\wt r_n(\wt\om)\xi_n)_{n\ge 1}$ are identically distributed.

Suppose now that $W$ is an $H$-cylindrical Brownian motion
(i.e, an $L^2(\R_+;H)$-isonormal process).
A function $\phi: \R_+\to H\otimes X$ is called an {\em elementary function} if it is a linear combination of
functions of the form $\one_{(s,t]}\otimes (h\otimes x)$ with $0\le s<t<\infty$, $h\in H$ and $x\in X$.
The stochastic integral with respect to $W$ of such a function is defined by putting
$$ \int_0^\infty \one_{(s,t]}\otimes (h\otimes x)\,dW := W(\one_{(s,t]}\otimes h)\otimes x$$
and extending this definition by linearity.

\begin{prop}\label{prop:type2-stochint}
Suppose that $X$ has type $2$ and let $\phi: \R_+\to H\otimes X$ be elementary.
Then
\[ \E \Big\| \int_0^\infty \phi\,dW \Big\|^2 \le \tau_{2,X}^2  \int_0^\infty \n \phi(t)\n_{\gamma(H,X)}^2\,dt.\]
\end{prop}
\begin{proof}
We may write
$$\phi = \sum_{n=1}^N \one_{(t_{n-1},t_n]}\otimes \sum_{j=1}^k h_j\otimes x_{jn}$$
for some fixed orthonormal system $(h_j)_{j=1}^k$ in $X$ and suitable
 $0\le t_0 < \dots < t_N<\infty$ and $x_{jn}\in X$.

Since the functions $(\one_{(t_{n-1},t_n)}\otimes h_j)/(t_n-t_{n-1})^{1/2}$ are orthonormal in $L^2(\R_+;H)$,
their images under $W$, denoted by $\gamma_{jn}$, form a Gaussian sequence. Hence,
by \eqref{eq:typerandom} and the type $2$ property,
\begin{align*}
 \E \Big\| \int_0^\infty \phi\,dW \Big\|^2
& =  \E \Big\|  \sum_{n=1}^N  \sum_{j=1}^k \gamma_{jn} \otimes [(t_n - t_{n-1})^{1/2} x_{jn}]\Big\n^2
\\ & =  \E \widetilde \E \Big\|   \sum_{n=1}^N \widetilde r_n \sum_{j=1}^k \gamma_{jn} \otimes [(t_n - t_{n-1})^{1/2} x_{jn}]\Big\n^2
\\ & \le \tau_{2,X}^2  \sum_{n=1}^N \E\Big\n \sum_{j=1}^k  \gamma_{jn} \otimes [(t_n - t_{n-1})^{1/2}x_{jn}]\Big\n^2
\\ & = \tau_{2,X}^2  \sum_{n=1}^N (t_n - t_{n-1})\E\Big\n \sum_{j=1}^k  \gamma_{jn} \otimes x_{jn}\Big\n^2
\\ & = \tau_{2,X}^2  \sum_{n=1}^N (t_n - t_{n-1})\Big\n \sum_{j=1}^k  h_j \otimes x_{jn}\Big\n_{\gamma(H,X)}^2
\\ & = \tau_{2,X}^2  \int_0^\infty \n \phi(t)\n_{\gamma(H,X)}^2\,dt.
\end{align*}
\end{proof}
The following proposition shows that there
is no hope of extending Proposition \ref{prop:type2-stochint} beyond the type $2$ case, even in the case $H = \R$
(in which case $W$ can be identified with a standard Brownian motion $B$ and $\gamma(H,X)$ with $X$).

As a preparation for the proof we recall the Kahane contraction principle,
which asserts that if $(\xi_n)_{n=1}^N$
is a sequence of independent and symmetric random variables, then for all scalar sequences $(a_n)_{n=1}^{N}$ we have
\begin{align*}
\E \Big\| \sum_{n=1}^{N} a_n \xi_n \Big\|^p & \leq
\max_{1\le n\le N}|a_N|^p\  \E \Big\| \sum_{n=1}^{N} \xi_n \Big\|^p
\end{align*}
for all $1\le p<\infty$.
Using this result together with the observation that $(\xi_n)_{n=1}^N$ and
$(r_n|\xi_n|)_{n=1}^N$ are identically distributed we see that
if $\inf_{1\le n\le N} \E |\xi_n| \ge \delta$, then
\begin{equation}\label{eq:contr}
\begin{aligned}
\wt \E \Big\n\sum_{n=1}^N \wt r_n x_n \Big\n^p & =
\wt \E \Big\n \E \sum_{n=1}^N \frac{\wt r_n |\xi_n|}{\E |\xi_n|} x_n \Big\n^p
\\ & \le \wt \E \E  \Big\n \sum_{n=1}^N \frac{\wt r_n |\xi_n|}{\E |\xi_n|} x_n \Big\n^p
\le \frac1{\delta^p} \E \Big\n\sum_{n=1}^N \xi_n x_n \Big\n^p.
\end{aligned}\end{equation}
In the case of standard Gaussian variables, note that
\begin{align}\label{eq:moment}
\E|\gamma| = \sqrt{2/\pi}
\end{align}

\begin{prop}
 If there exists a constant $C\ge 0$ such that for all elementary functions $\phi: \R_+\to X$
we have
\[ \E \Big\| \int_0^\infty \phi\,dB \Big\|^2 \le C^2  \int_0^\infty \n \phi(t)\n^2\,dt,\]
then $X$ has type $2$.
\end{prop}
\begin{proof}
 Fix  $x_1,\dots,x_N\in X$ and consider the function
$\phi = \sum_{n=1}^N \one_{(n-1,n]}\otimes x_n.$ If a constant $C\ge 0$ with the above property
exists, then, using that the increments $B_n - B_{n-1}$ are standard Gaussian and independent,
\begin{align*}
 \E \Big\n\sum_{n=1}^N \gamma_n x_n \Big\n^2 & = \E \Big\| \sum_{n=1}^N (B_n - B_{n-1}) x_n \Big\|^2
\\ & = \E \Big\| \int_0^\infty \phi\,dB \Big\|^2
 \le C^2  \int_0^\infty \n \phi(t)\n^2\,dt = C^2 \sum_{n=1}^N  \| x_n\|^2.
\end{align*}
This proves that $X$ has {\em Gaussian type $2$}, with Gaussian type $2$ constant $\tau_{2,X}^\gamma \le C$.
By \eqref{eq:contr} and \eqref{eq:moment}, this implies
that $X$ has type $2$ with $\tau_{2,X} \le C\sqrt{\pi /2}$.
\end{proof}
Further examples may be found in \cite{RS, Yo}.

\subsection{Random integrands}\label{subsec:random} If one tries to extend the above proof to
the case of a random integrand,
one sees that the type $2$ property does not suffice. Indeed, the coupling between the integrand and
$W$ destroys the Gaussianity. However, the martingale structure is retained, and this can be
exploited to make a variation of the argument work under a slightly stronger assumption
on the Banach space $X$, viz. that it has
martingale type $2$.

\begin{defn} Let $p\in [1,2]$.
 A Banach space $X$ has {\em martingale type $p$} if there exists a constant $\mu\ge 0$ such that for all
all finite $X$-valued martingale difference sequences $(d_n)_{n=1}^N$ we have
\begin{equation}\label{eq:martingaletype}
\E\Big\| \sum_{n=1}^N d_n\Big\|^p \le \mu^p\sum_{n=1}^N\E \| d_n\|^p.
\end{equation}
\end{defn}
The least admissible constant in this definition is denoted by $\mu_{p,X}$.

\begin{ex} Here are some examples.
\begin{itemize}
\item Every Banach space has martingale type $1$.
 \item
Every Hilbert space has martingale type $2$.
 \item
Every $L^p(\mu)$ space, $1\leq p<\infty$, has martingale type $p\wedge 2$.
\end{itemize}
\end{ex}
Since every Gaussian sequence is a martingale difference sequence, we see immediately that
every Banach space with martingale type $p$ has type $p$, with constant $\tau_{p,X}\le \mu_{p,X}$.

Suppose now that an $H$-cylindrical Brownian motion $W$ is given. We shall denote by
 $(\F_t)_{t\ge 0}$ the filtration induced by $W$, i.e., $\F_t$ is the $\sigma$-algebra generated
by all random variables $W(f)$ with $f\in L^2(0,t;H)$. The following lemma is proved by a standard
monotone class argument.

\begin{lem}
 If the functions $f_1,\dots, f_k\in L^2(\R_+;H)$ have support in $[t,\infty)$, then $(W(f_1), \dots, W(f_k))$
is independent of $\F_t$.
\end{lem}
More generally, we could consider any filtration with the property stated in the lemma.

Let $\phi:\R_+\times\Omega\to H\otimes X$ be an {\em adapted elementary process}. By this we mean that
$\phi$ is a linear combination of processes of the form
$$ \one_{(s,t]\times F} \otimes (h\otimes x)$$
with $0\le s<t$, $F\in \F_s$, $h\in H$, and $x\in X$.
The stochastic integral of $\phi$ with respect to $W$ is then defined by putting
$$\int_0^\infty \one_{(s,t]\times F} (h\otimes x) \, dW : = \one_{F} W(\one_{(s,t]}\otimes h)\otimes x$$
and extending this definition by linearity.

\begin{thm}\label{thm:Mtype2-stochint} Suppose that the Banach space $X$ has martingale type $2$ and let
$\phi:\R_+\times\Omega\to H\otimes X$ be an adapted elementary process.
Then
$$ \E \Big\n \int_0^\infty \phi\,dW\Big\n^2 \le \mu_{2,X}^2 \E \int_0^\infty \n \phi_t\n_{\gamma(H,X)}^2\,dt.$$
\end{thm}

\begin{proof}
By assumption we may represent $\phi$ as
\begin{equation}\label{eq:phi}
 \phi = \sum_{n=1}^N \one_{(t_{n-1}, t_n]} \sum_{m=1}^M \one_{F_{mn}} \otimes
\sum_{j=1}^k h_j \otimes x_{jmn}.
\end{equation}
Here, $(h_j)_{j=1}^k$ is an orthonormal system in $H$,
for each $1\le n\le N$ the sets $F_{mn}$, $1\le m\le M$, are disjoint
and belong to $\F_{t_{n-1}}$, and the vectors $x_{jmn}$ are taken from $X$.
Then
$$ \int_0^\infty \phi \, dW
 = \sum_{n=1}^N\sum_{m=1}^M\sum_{j=1}^k    \one_{F_{mn}} W(\one_{(t_{n-1}, t_n]}\otimes h_j)\otimes x_{jmn}.$$
As before, the images under $W$ of the functions
$(\one_{(t_{n-1}, t_n]}\otimes h_j)/{(t_n - t_{n-1})^{1/2}}$, which we denote by
$\gamma_{jn}$, form a Gaussian sequence.
The random variables
$$d_{n} := (t_n - t_{n-1})^{1/2} \sum_{m=1}^M \sum_{j=1}^k\one_{F_{mn}} \gamma_{jn} \otimes x_{jmn}
$$ form a martingale difference sequence $(d_n)_{n=1}^N$ with
respect to $(\F_{t_n})_{n=0}^N$. To see this,  note that $d_n$ is $\F_{t_n}$-measurable and
$F_{nm}\in \F_{t_{n-1}}$, and therefore
$$\E(\one_{F_{mn}} \gamma_{jn}|\F_{t_{n-1}}) = \one_{F_{mn}} \E \gamma_{jn} = 0$$
since $\gamma_{jn}$ is independent of $\F_{t_{n-1}}$.
Using the martingale type $2$ property of
$X$, the lemma, and the disjointness of the sets $F_{1n},\dots,F_{Mn}$,
we may now estimate
\begin{align*}
\E \Big\n\int_0^\infty \phi \, dW\Big\n^2
&  = \E \Big\n  \sum_{n=1}^N (t_n - t_{n-1})^{1/2}\sum_{m=1}^M \sum_{j=1}^k   \one_{F_{mn}} \gamma_{jn} \otimes x_{jmn}\Big\n^2
\\ &  = \E \Big\n \sum_{n=1}^N d_{n} \Big\n^2
 \le \mu_{2,X}^2 \sum_{n=1}^N\E \n d_{n} \n^2
\\ & = \mu_{2,X}^2  \sum_{n=1}^N (t_n - t_{n-1})\sum_{m=1}^M \E\one_{F_{mn}} \E \Big\n\sum_{j=1}^k \gamma_{jn} x_{jmn}\Big\n^2
\\ & = \mu_{2,X}^2 \sum_{n=1}^N (t_n - t_{n-1})\sum_{m=1}^M \E \one_{F_{mn}} \Big\n\sum_{j=1}^k h_j\otimes x_{jmn}\Big\n_{\gamma(H,X)}^2
\\ & = \mu_{2,X}^2 \E \int_0^\infty \n \phi_t\n_{\gamma(H,X)}^2\,dt.
\end{align*}
\end{proof}
By Doob's inequality, this improves to the maximal inequality
$$ \E \sup_{t\geq 0} \Big\n \int_0^t \phi\,dW\Big\n^2 \le 4\mu_{2,X}^2 \E \int_0^\infty \n \phi_t\n_{\gamma(H,X)}^2\,dt.$$
From here, it is a routine density argument
to extend the stochastic integral to arbitrary progressively measurable processes
$\phi:\R_+\times\Om\to\gamma(H,X)$ that satisfy $\E  \int_0^\infty \n \phi_t\n_{\gamma(H,X)}^2\,dt<\infty$;
the process $ t\mapsto \int_0^t \phi\,dW$
is then a continuous martingale.
Then, the usual stopping time techniques apply to extend the integral to progressively measurable processes
satisfying $\int_0^\infty \n \phi_t\n_{\gamma(H,X)}^2\,dt < \infty$ almost surely.

The following version of Burkholder's inequality holds:

\begin{thm}\label{thm:Burk-Mtype2} Let $X$ have martingale type $2$. Then for any strongly measurable
adapted process
$\phi:\R_+\times\Omega\to \gamma(H,X)$ and $0< p<\infty$,
\begin{align}\label{eq:BDG}
\E \sup_{t\geq 0}\Big\|  \int_0^t \phi\,dW \Big\|^p \le C_{p,X}^p \|\phi\|_{L^p(\O;L^2(\R_+;\gamma(H,X)))}^p.
\end{align}
\end{thm}
For $p\ge 2$ this result is due to Dettweiler \cite{Dett91}
who gave a proof based on a martingale version of Rosenthal's inequality.
A particularly simple proof, based on a good-$\lambda$ inequality, was obtained by Ondrej\'at \cite{Ondrejat04}.
Both proofs produce non-optimal constants $C_{p,X}$ as $p\to \infty$. A proof with the optimal constant
$$C_{p,X} \le C_X \sqrt{p}, \quad p\ge 2,$$ was obtained by Seidler \cite{Seidler} using
square function techniques in combination with a maximal inequality for
discrete-time martingales due to Pinelis \cite{Pin}.

The drawback of the martingale type $2$ theory is not so much the fact that the class of spaces to which it applies
is rather limited (e.g., it applies to $L^p$-spaces only for $p\in [2,\infty)$) but rather the fact that the
inequalities of Theorems \ref{thm:Mtype2-stochint} and \ref{thm:Burk-Mtype2} are not sharp.
In applications to parabolic SPDE, this
lack of sharpness prevents one from proving the sharp endpoint inequalities needed for maximal regularity
of mild solutions.
As we will outline next, the theory of stochastic integration in UMD spaces does produce the sharp estimates
that are needed for this purpose.

\section{Stochastic integration in UMD spaces}

\subsection{Deterministic integrands}

Let $X$ be an arbitrary Banach space
and $W$ be an $H$-cylindrical Brownian motion.
For an elementary function $\phi: \R_+\to H\otimes X$
we define the stochastic integral $\int_0^\infty\phi\,dW$
as before. The following proposition provides a {\em two-sided} estimate
for the $L^p$-norms of this integral.
As a preliminary observation we note that $\phi$, being an elementary
function, defines an element in the algebraic tensor product
$ L^2(\R_+) \otimes (H\otimes X)$. In view of the linear isomorphism of vector spaces
\begin{equation}\label{eq:identif}
 L^2(\R_+) \otimes (H\otimes X)\simeq (L^2(\R_+) \otimes H)\otimes X
\end{equation}
we may view $\phi$ as an element of  $(L^2(\R_+) \otimes H)\otimes X$.
Identifying $L^2(\R_+)\otimes H$ with a dense subspace of $L^2(\R_+;H)$, we may
view $\phi$ as an element in $\gamma(L^2(\R_+;H), X)$.

\begin{prop}[It\^o isometry]\label{prop:Ito}
Let $X$ be a Banach space and let $p\in [1, \infty)$.
For all elementary functions $\phi:\R_+\to H\otimes X$ we have
\[\E \Big\| \int_0^\infty \phi\,dW\Big\|^p = \n \phi\n_{\gamma^p(L^2(\R_+;H),X)}^p.\]
\end{prop}
\begin{proof}
Representing $\phi$ as in the proof of Proposition \ref{prop:type2-stochint} and using the notations
introduced there, we have
\begin{align*}
\E \Big\n \int_0^\infty \phi\,dW \Big\n^p
& = \E \Big\n \sum_{n=1}^N  \sum_{j=1}^k \gamma_{jn} \otimes (t_n - t_{n-1})^{1/2}  x_{jn}\Big\n^p
\\ & = \Big\n \sum_{n=1}^N  \sum_{j=1}^k f_{jn} \otimes (t_n - t_{n-1})^{1/2}x_{jn}\Big\n_{\gamma^p(L^2(\R_+;H),X)}^p
\\ & =  \n \phi\n_{\gamma^p(L^2(\R_+;H);X)}^p,
\end{align*}
where we used that the functions $f_{jn}:= (\one_{(t_{n-1},t_n]}\otimes h_j)/(t_n - t_{n-1})^{1/2}$
are orthonormal in $L^2(\R_+;H)$ and satisfy $\sum_{n=1}^N \sum_{j=1}^k f_{jn}\otimes (t_n - t_{n-1})^{1/2}
x_{jn} = \phi$.
\end{proof}
By a density argument, the mapping $\phi\mapsto \int_0^\infty \phi\,dW$ extends to an isometry from
$\gamma^p(L^2(\R_+;H),X)$ into $L^p(\Omega;X).$

Combining the estimates of Propositions \ref{prop:type2-stochint} and \ref{prop:Ito} under the assumption
that $X$ have type $2$,
we obtain the inequality
$$ \n \phi\n_{L^2(\R_+;\gamma(H,X))} \le \tau_{2,X} \n \phi\n_{\gamma(L^2(\R_+;H);X)}
$$ for elementary functions $\phi$. This
implies that if $X$ has type $2$, then
the natural identification made in \eqref{eq:identif} extends to a bounded inclusion
$$ L^2(\R_+;\gamma(H,X)) \embed \gamma(L^2(\R_+;H);X)$$
of norm at most $\tau_{2,X}$.
For further results along this line we refer the reader to \cite{KNVW, vNWe, RS}.

\subsection{UMD spaces}

Next we show that it is possible to extend Proposition \ref{prop:Ito} to random integrands
if $X$ is a UMD space.
We start with a brief introduction of this class of Banach spaces.

\begin{defn}\label{def:UMD}
A Banach space $X$ is called a {\em UMD space} if for some $p\in (1, \infty)$ (equivalently,
for all $p\in (1,\infty)$) there is a constant $\beta\ge 0$ such that for all
$X$-valued $L^p$-martingale difference sequences $(d_n)_{n\geq 1}$ and all signs $(\epsilon_n)_{n\geq 1}$ one has
\begin{equation}\label{eq:UMD}
 \E \Big\| \sum_{n=1}^N \epsilon_n d_n\Big\|^p \le \beta^p  \E \Big\| \sum_{n=1}^N  d_n\Big\|^p, \quad \forall N\geq 1.
\end{equation}
\end{defn}
The least admissible constant in this definition is called the {\em UMD$_p$-constant} of $X$ and is
denoted by $\beta_{p,X}$.
It is by no means obvious that once the UMD property holds for one $p\in (1,\infty)$, then it holds for all
$p\in (1,\infty)$; this seems to have been first observed by Pisier, whose proof was outlined in \cite{Mau}.
A more systematic proof based on martingale decompositions can be found in \cite{Bu1} and the survey paper \cite{Burk01}.

\begin{ex} Let us provide some examples of UMD spaces. Fix $p,p'\in (1, \infty)$ such that $\frac1p+\frac{1}{p'}=1$.
\begin{itemize}
 \item
 Every Hilbert space $H$ is a UMD space (with $\tau_{p,H} = \max\big\{p,p'\big\}$).
 \item
 The spaces $L^p(\mu)$, $1<p<\infty$, are UMD spaces (with $\beta_{p,L^p(\mu)} = \max\big\{p,p'\big\}$).
 More generally, if $X$ is a UMD space, then $L^p(\mu;X)$, $1<p<\infty$, is a UMD space (with $\beta_{p,L^p(\mu;X)} = \beta_{p,X}$)
 \item $X$ is a UMD space if and only $X\s$ is a UMD space (with $\beta_{p,X} = \beta_{p',X\s}$).
 \item Every Banach space which is isomorphic to a closed subspace or a quotient of a UMD space is a UMD space.
\end{itemize}
\end{ex}
By applying \eqref{eq:UMD} to the martingale difference sequence $(\epsilon_n d_n)_{n\ge 1}$ one obtains
the reverse estimate
\begin{equation}\label{eq:UMDreverse}
 \E \Big\| \sum_{n=1}^N  d_n\Big\|^p \le
\beta_{p,X}^p \E \Big\| \sum_{n=1}^N \epsilon_n d_n\Big\|^p , \quad \forall N\geq 1.
\end{equation}
If $(r_n)_{n\geq 1}$ is a Rademacher sequence which is independent of $(d_n)_{n\geq 1}$,
then \eqref{eq:UMD} and \eqref{eq:UMDreverse} easily imply the two-sided randomised inequality
\begin{equation}\label{eq:UMDrandom}
\frac1{\beta_{p,X}^p}\E\Big\| \sum_{n=1}^N d_n \Big\|^p\leq \E \Big\| \sum_{n=1}^N r_n d_n \Big\|
\leq \beta_{p,X}^p\E \Big\| \sum_{n=1}^N d_n \Big\|^p, \quad \forall N\geq 1,
\end{equation}
In \cite{Ga2} the lower and upper
estimates in \eqref{eq:UMDrandom} were studied for their own sake (see also Remark \ref{rem:UMD-}).

We include the simple observation that within the class of UMD spaces, the notions
of type and martingale type are equivalent (see \cite{Brz1}).

\begin{prop}
Let $p\in [1, 2]$.
If $X$ is a UMD space with type $p$, then $X$ has martingale type $p$ and $\mu_{p,X}\le \beta_{p,X}\tau_{p,X}$.
\end{prop}
\begin{proof}
Let $(\wt r_n)_{n\geq 1}$ be a Rademacher sequence on another probability space $(\wt\Om,\wt \P)$.
By \eqref{eq:UMDrandom} and Fubini's theorem,
\[
\E \Big\| \sum_{n=1}^N d_n\Big\|^p \leq \beta_{p,X}^p \E\wt \E\Big\| \sum_{n=1}^N \wt r_n d_n\Big\|^p
 \leq \beta_{p,X}^p \tau_{p,X}^p \E \sum_{n=1}^N \|d_n\|^p.
\]
\end{proof}

\subsection{Decoupling}

The extension of Proposition \ref{prop:Ito} to adapted elementary processes will be achieved by
means of a decoupling technique, which allows us to replace the cylindrical Brownian motion $W$
by an independent copy $\wt W$ on a second probability space $\wt \Om$.
With respect to $\wt W$, we may estimate the $L^p$-norms
path-by-path with respect to $\Omega$. The UMD property will provide the relevant estimates
for the decoupled integral in terms of the original integral and vice versa.

We begin with a decoupling inequality for martingale transforms due to McConnell \cite{MC}.
The setting is as follows. We are given a probability space
$(\Omega, \P)$ with filtration $\F$, and independent copies $(\wt \Om, \wt \P)$ and $\wt \F$.
We identify $\F$ and $\wt \F$ with the filtrations on $\Om\times \wt \Om$ given by
$\F\times\{\emptyset,\wt\Om\}$ and $\{\emptyset,\Om\}\times\wt \F$, respectively. In a similar way,
random variables $\xi$ and $\wt \xi$ on $\Omega$ and $\wt \Omega$ are identified with the random variables
$\xi(\om,\wt \om):= \xi(\om)$ and $\wt\xi(\om,\wt \om):= \wt \xi(\wt\om)$
on $\Om\times\wt \Om$, respectively.

\begin{thm}\label{thm:decineq}
Let $X$ be a UMD space and let $p\in (1, \infty)$. Let $(\eta_n)_{n\geq 1}$
be an $\F$-adapted sequence of centered random variables in $L^p(\Omega)$
such that for each $n\geq 1$, $\eta_n$
is independent of $\F_{n-1}$. Let $(\wt{\eta}_n)_{n\geq 1}$ be an independent $\wt \F$-adapted
copy of this sequence in $L^p(\wt{\Omega};X)$.
Finally, let $(v_n)_{n\geq 1}$ be an
$\F$-predictable sequence in $L^\infty(\Om;X)$.
Then, for all $N\geq 1$,
\begin{equation}\label{eq:gendecoupling}
\frac1{\beta_{p,X}^{p}} \E\wt\E\Bigl\|\sum_{n=1}^N v_n \wt{\eta}_n \Bigr\|^p
\leq \E\wt\E\Bigl\|\sum_{n=1}^N v_n \eta_n \Bigr\|^p \leq \beta_{p,X}^p
  \E\wt\E\Bigl\|\sum_{n=1}^N v_n \wt{\eta}_n \Bigr\|^p.
\end{equation}
\end{thm}
This decoupling inequality was further extended in \cite{MC} to more general martingale
difference sequences.

\begin{proof}
The functions $\eta_n:\Omega\to X$ and $\wt{\eta}_{n}: \wt\Omega\to X$ will be interpreted as
functions on $\Omega\times\wt{\Omega}$ by considering
$(\omega,\wt{\omega})\mapsto \eta_n(\omega)$ and
$(\omega,\wt{\omega})\mapsto \wt{\eta}_n(\wt{\omega})$, respectively.

For $n=1,\dots,N$ define
$$ d_{2n-1} := \tfrac12 v_n(\eta_n + \wt{\eta}_n)\quad \hbox{and}\quad
     d_{2n} := \tfrac12 v_n(\eta_n - \wt{\eta}_n).$$
We claim that
$(d_j)_{j=1}^{2N}$ is an $L^p$-martingale difference sequence with respect to the
filtration $({\mathcal G}_j)_{j=1}^{2N}$, where for $n\geq 1$,
\begin{align*}
{\mathcal G}_{2n-1}  = \sigma(\F_{{n-1}}\times\wt\F_{{n-1}},
\eta_n+\wt{\eta}_n)
\quad \hbox{and}\quad
{\mathcal G}_{2n} = \F_{n}\times\wt\F_{n},
\end{align*}
with $\F_{n}\times\wt\F_{n}$ denoting the product $\sigma$-algebra.
Clearly, $(d_n)_{n=1}^{2 N}$ is $(\mathcal{G}_n)_{n=1}^{2N}$-adapted. For $n=1,
\ldots, N$,
\[\CE{d_{2n+1}}{\mathcal{G}_{2n}}
= \tfrac12 v_{n+1}\CE{\eta_{n+1}+\wt{\eta}_{n+1}}{\mathcal{G}_{2n}}
=\tfrac12 v_{n+1}(\E\eta_{n+1}+\wt{\E}\wt{\eta}_{n+1})=0,\]
since $\eta_{n+1}$ and $\wt{\eta}_{n+1}$ are independent of
$\mathcal{G}_{2n}$ and centered.
For $n=1, \ldots, N$,
\[\CE{d_{2n}}{\mathcal{G}_{2n-1}} =\tfrac12 v_n\CE{\eta_n - \wt{\eta}_n}{ \mathcal{G}_{2n-1}}
\stackrel{\rm(i)}{=}  \tfrac12 v_n\CE{\eta_n - \wt{\eta}_n}{\eta_n
+\wt{\eta}_n}\stackrel{\rm(ii)}{=} 0.\]
Here (i) follows from the independence
of $\sigma(\eta_n,\wt{\eta}_n)$ and $\F_{{n-1}}\times\wt\F_{{n-1}}$.
For the identity (ii) let $B\subseteq X$ be a Borel set.
Let $\nu$ and $\wt{\nu}$ denote the image measure of $\eta_n$ and $\wt{\eta}_n$ on $\mathcal{B}(X)$, respectively.
Then $\nu = \wt{\nu}$ and therefore
\begin{align*}
\E\wt\E \one_{\{\eta_n+\wt{\eta}_n\in B\}} \, \eta_n  &=
\int_X \int_X \one_{\{x+y\in B\}} \, x \, d\nu(x) d\nu(y) \\ & =
\int_X \int_X \one_{\{x+y\in B\}} y \, d\nu(y) d\nu(x)
=  \E\wt\E \one_{\{\eta_n+\wt{\eta}_n\in B\}} \, \wt{\eta}_n,
\end{align*}
which gives (ii) and also finishes the proof of the claim.

Now since
\[
\sum_{n=1}^N v_n\eta_n  = \sum_{j=1}^{2N} d_j\quad\hbox{ and }\quad
\sum_{n=1}^N v_n\wt{\eta}_n  = \sum_{j=1}^{2N} (-1)^{j+1}d_j,
\]
the result follows from the UMD property applied to the sequences $(d_j)_{j=1}^{2N}$ and
$((-1)^{j+1}d_j)_{j=1}^{2N}$.
\end{proof}

\subsection{Random integrands} We are now in a position to prove sharp estimates for the stochastic
integrals of adapted elementary processes. Similar to what we did in the case of elementary functions, we will
identify an adapted elementary process with an element of
$$ (L^2(\R_+) \otimes L^p(\Om)) \otimes (H\otimes X) \simeq
L^p(\Om) \otimes((L^2(\R_+)\otimes H)\otimes X).$$
In the next theorem we identify the right-hand side with a dense subspace
of $L^p(\Omega;\gamma(L^2(\R_+;H),X))$.

\begin{thm}[It\^o isomorphism]\label{thm:Itoisomorph}
Let $X$ be a UMD space and let $p\in (1, \infty)$.
For all adapted elementary processes $\phi:\R_+\times\Omega\to H\otimes X$ we have
\begin{align*} \frac1{\beta_{p,X}}\|\phi\|_{L^p(\Omega;\gamma^p(L^2(\R_+;H),X))}
& \leq \Big\| \int_0^\infty \phi\,dW \Big\|_{L^p(\Omega;X)}
\\ & \leq  \beta_{p,X}\|\phi\|_{L^p(\Omega;\gamma^p(L^2(\R_+;H),X))}.
\end{align*}
\end{thm}
\begin{proof}
Let $\wt W$ be an $H$-cylindrical Brownian motion on a probability space $\wt \Omega$. As before
we may view $W$ and $\wt W$ as independent $H$-cylindrical Brownian motions on $\Om\times\wt \Om$.

We may represent $\phi$ as in \eqref{eq:phi}, i.e.,
$$
 \phi = \sum_{n=1}^N \one_{(t_{n-1}, t_n]} \sum_{m=1}^M \one_{F_{mn}} \otimes
\sum_{j=1}^k h_j \otimes x_{jmn},$$
where $(h_j)_{j=1}^k$ is orthonormal in $H$,
for each $1\le n\le N$ the sets $F_{mn}$, $1\le m\le M$, are disjoint
and belong to $\F_{t_{n-1}}$, and the vectors $x_{jmn}$ are taken from $X$.
We view $\phi$ as being defined on $\Om\times\wt\Om$.

Define, for $1\le j\le k$ and $1\le n\le N$,
$$
    \eta_{jn} := W(\one_{(t_{n-1},t_n]}\otimes h_{j}), \quad   \wt\eta_{jn} := \wt W(\one_{(t_{n-1},t_n]}\otimes h_{j}),
$$ and
$$
   v_{jn}  := \sum_{m=1}^M 1_{F_{mn}}\otimes x_{jmn}.
$$
With these notations,
$$ \int_0^T \phi\,d W = \sum_{n=1}^N\sum_{j=1}^k v_{jn}\eta_{jn}, \quad
\int_0^T \phi\,d \wt W = \sum_{n=1}^N\sum_{j=1}^k v_{jn} \wt \eta_{jn}.
$$
We consider the filtration $(\F_{jn})$,
where
$$\F_{jn} = \sigma(\F_{t_{n-1}}, \eta_{1n}, \dots, \eta_{jn});$$
the indices $(jn)$ are ordered
lexicographically by the rule $(j',n')\le (j,n) \Longleftrightarrow
n'< n$ or [$n'=n\ \& \ j'\le j$].
By Theorem \ref{thm:decineq},
\begin{align*}
\ & \frac1{\beta_{p,X}} \Bigl\|\sum_{n=1}^N\sum_{j=1}^k v_{jn}\wt \eta_{jn}  \Bigr\|_{L^p(\Omega\times\wt{\Omega};X)}
\\ & \qquad \qquad \leq
\Bigl\|\sum_{n=1}^N\sum_{j=1}^k v_{jn}\eta_{jn} \Bigr\|_{L^p(\Omega;X)} \leq \beta_{p,X}
\Bigl\|\sum_{n=1}^N\sum_{j=1}^k v_{jn}\wt\eta_{jn} \Bigr\|_{L^p(\Omega\times\wt{\Omega};X)}.
\end{align*}
On the other hand, by Proposition \ref{prop:Ito}, for each $\omega\in \Omega$ we have
\[\Bigl\|\sum_{n=1}^N v_n(\omega)\wt{\eta}_n   \Bigr\|_{L^p(\wt{\Omega};X)} =
\|{\phi}(\omega)\|_{\gamma^p(L^2(\R_+;H), X)}.\]
Therefore, by Fubini's theorem,
\[\Bigl\|\sum_{n=1}^N\sum_{j=1}^k v_{jn}\eta_{jn}  \Bigr\|_{L^p(\Omega\times \wt{\Omega};X)}
= \|{\phi}\|_{L^p(\Omega;\gamma^p(L^2(\R_+;H), X))}.\]
\end{proof}
By an application of Doob's inequality, for $1<p<\infty$ we obtain the equivalence of
norms
\begin{equation}\label{eq:UMD-BDG}
\begin{aligned}
\ & \frac1{\beta_{p,X}} \|\phi\|_{L^p(\Omega;\gamma^p(L^2(\R_+;H),X))}
\\ & \qquad \leq \Big(\E \sup_{t\geq 0} \Big\|\int_0^t \phi\,dW \Big\|^p\Big)^{1/p}
 \leq  \frac{p }{p-1} \beta_{p,X}\|\phi\|_{L^p(\Omega;\gamma^p(L^2(\R_+;H),X))}.
\end{aligned}
\end{equation}
By a standard application of Lenglart's inequality \cite{Lenglart}, this equivalence
extends to all exponents $0<p<\infty$ with different constants
(see Remark \ref{rem:kleinepLenglart} below).
This yields the UMD analogue of the
Burkholder inequality of Theorem \ref{thm:Burk-Mtype2}. It is interesting to observe that
no additional argument is needed to pass from the case $p=2$ to the case $1<p<\infty$;
the result for $1<p<\infty$ is obtained right away from the decoupling inequalities.

By Theorem \ref{thm:Itoisomorph}, the stochastic integral can be extended to the closure in
$L^p(\Om;\gamma(L^2(\R_+;H),X))$ of all adapted elementary processes. We shall denote this closure by
$L^p_\F(\Om;\gamma(L^2(\R_+;H),X))$. In this way
the stochastic integral defines an isomorphic embedding
\begin{align}\label{eq:I-isom}
I:L^p_\F(\Om;\gamma(L^2(\R_+;H),X)) \to L^p(\Om,\F_\infty;X).
\end{align}
Moreover, by \eqref{eq:UMD-BDG}, the indefinite stochastic
integral defines an isomorphic embedding of $L^p_\F(\Om;\gamma(L^2(\R_+;H),X))$
into $L^p(\Om;C_{\rm b}(\R_+;X))$.

In the special case of the augmented filtration $\F^W$ generated by $W$,
the isomorphic embedding \eqref{eq:I-isom} is actually onto (pass to the
limit $T\to\infty$ in the corresponding result for finite time intervals in \cite[Theorem 3.5]{NVW1})
and we obtain an isomorphism of Banach spaces
$$ I : L^p_{\F^W}(\Om;\gamma(L^2(\R_+;H),X)) \simeq L^p(\Om,\F_\infty^W;X).$$
This result contains a martingale representation theorem: every $\F_\infty^W$-meas\-urable
random variable in $L^p(\Om;X)$ is the stochastic integral of a suitable element of
$L^p_{\F^W}(\Om;\gamma(L^2(\R_+;H),X))$.

We continue with a description of
$L^p_\F(\Om;\gamma(L^2(\R_+;H),X))$ (the definition of which extends
to $p\in [0,\infty)$ in the obvious way). A proof can be found in
\cite[Proposition 2.10]{NVW1}.

\begin{prop}
Let $p\in [0, \infty)$. For an element $\phi\in L^p(\Omega;\gamma(L^2(\R_+;H),X))$ the following assertions are
equivalent:
\begin{enumerate}[\rm(1)]
 \item $\phi\in L^p_{\F}(\Omega;\gamma(L^2(\R_+;H),X))$;
 \item the random variable $\lb\phi(\one_{[0,t]}f),x\s\rb\in L^p(\Omega)$ is $\F_t$-measurable for all
$t\in \R_+$, $f\in L^2(\R_+;H)$, and $x\s\in X\s$.
\end{enumerate}
\end{prop}
In particular
if $\phi: \R_+\times \Omega\to \calL(H,X)$ is {\em $H$-strongly measurable and adapted},
in the sense that for all
$h\in H$ the $X$-valued process $\phi h:\R_+\times\Om\to X$ is strongly measurable and adapted,
then $\phi\in L^p(\Omega;\gamma(L^2(\R_+;H),X))$ implies $\phi\in L_\F^p(\Omega;\gamma(L^2(\R_+;H),X))$.
Indeed, in that case, for all $h\in H$ and $x\s\in X\s$ the process $[ h, \phi\s x\s]$ is measurable and adapted. Since
$\phi:\Om\to\gamma(L^2(\R_+;H),X)$ is strongly measurable and
elements in
$\gamma(L^2(\R_+;H)$ are separably supported (see \cite[Section 3]{NeeCMA}), we may assume that $H$
is separable, and then the Pettis measurability theorem implies that the $H$-valued process $\phi\s x\s$ is
strongly measurable and adapted. Passing to a progressively measurable
version of $\phi\s x\s$ (see \cite{OndSei} for a short existence proof), we see that
$\lb\phi(\one_{[0,t]}f),x\s\rb$ is equal almost surely to a strongly $\F_t$-measurable random variable.

In the special case $X = L^q(\mu)$ with $1<q<\infty$,
combination of \eqref{eq:UMD-BDG} with \eqref{eq:gammaLp} gives the following two-sided inequality for a measurable and adapted
processes $\phi:\R_+\times\O\to L^q(\mu;H)$: if $\phi\in L^p(\O;L^q(\mu;L^2(\R_+;H)))$ for some $0<p<\infty$, then
\begin{align}\label{eq:itoisom}
\E \sup_{t\geq 0} \Big\|\int_0^\cdot \phi\,dW \Big\|_{L^q(\mu)}^p \eqsim_{p,q} \E\|\phi\|_{L^q(\mu;L^2(\R_+;H)))}^p.
\end{align}

The next step in the construction of the UMD stochastic integral consists in a localisation argument.
The process
$$ \zeta := \int_0^\cdot \phi\,dW$$
is a continuous martingale,
and by standard stopping time techniques (see \cite[Lemma 4.6]{RY})
one proves the following inequalities, valid for all
$\delta>0$ and $\varepsilon>0$:
\begin{align} \label{probineq1E}
\ & \P\big(\|\zeta\|_{C_{\rm b}(\R_+;X)}>\varepsilon\big)
\\ & \qquad  \leq \varepsilon^{-p}
{C_{p,X}^p \E(\delta^p \wedge \|\phi\|_{\gamma^p(L^2(\R_+;H),X)}^p)} + \P\bigl(\|\phi\|_{\gamma^p(L^2(\R_+;H),X)}\geq
\delta\bigr),
\end{align}
where
$C_{p,X} = \frac{p}{p-1} \beta_{p,X}$, and
\begin{align}\label{probineq2E}
\ & \P\bigl(\|\phi\|_{\gamma^p(L^2(\R_+;H),X)}>\varepsilon\bigr)
\\ & \qquad \leq \varepsilon^{-p}{\beta_{p,X}^p \E(\delta^p \wedge \|\zeta\|_{C_{\rm b}(\R_+;X)}^p)}
+ \P\big(\|\zeta\|_{C_{\rm b}(\R_+;X)}\geq \delta\big).
\end{align}
A direct consequence is that the stochastic integral $I: \phi\mapsto \int_0^\cdot \phi\,dW$
uniquely extends to a continuous linear embedding
\begin{align}\label{eq:L0-isom}
 I: L^0_{\F}(\Omega;\gamma(L^2(\R_+;H),X)) \to L^0(\Omega;C_{\rm b}(\R_+;X))).
\end{align}
For the details we refer to  \cite{NVW1}. We call
$I\phi$ the stochastic integral of $\phi$ with respect to $W$ and write
$$ \int_0^t \phi\,dW = I\phi(t), \quad t\ge 0, \ \phi\in L^0_{\F}(\Omega;\gamma(L^2(\R_+;H),X)) .$$

\begin{rem}\label{rem:kleinepLenglart}
Fix $p\in (1, \infty)$ and $0<q<p$.
Taking $\varepsilon = \delta$ in the above estimates
and integrating with
respect to $d \varepsilon^q$ one obtains that (see \cite[Proposition 4.7]{RY} for a similar argument)
\begin{align*}
\ & \Big(\frac{p-q}{\beta_{p,X} p}\Big)^{1/q}\|\phi\|_{L^q(\Omega;\gamma^p(L^2(\R_+;H),X))}
\\ & \qquad \leq \Big(\E \sup_{t\geq 0} \Big\|\int_0^t \phi\,dW \Big\|^q\Big)^{1/q}
 \leq  \Big(\frac{p C_{p,X}}{p-q}\Big)^{1/q} \|\phi\|_{L^q(\Omega;\gamma^p(L^2(\R_+;H),X))}.
\end{align*}
\end{rem}

Up to this point we have set up the abstract stochastic integral by a density argument, starting from
adapted elementary processes. The next result, taken from \cite[Theorem 4.1]{OndrVer2013},  gives
a criterion which enables one to decide whether
a given operator-valued stochastic process belongs to the closure of the adapted elementary processes.
Earlier versions of this result, as well as related characterisations, can be found in \cite{NVWco, NVW1}.

\begin{thm}\label{thm:bddpaths}
Let $X$ be a UMD Banach space. Let $\phi:\R_+\times\Om\to \calL(H,X)$ be an
$H$-strongly measurable adapted
process such that $\phi\s x\s \in L^0(\Om;L^2(\R_+;H))$ for all
$x\s\in X\s$. Let
$\zeta:\R_+\times\Om \to X$ be a process whose paths are almost surely bounded.
If for all $x^*\in X^*$ almost surely, one has
\[\int_0^t \phi^* x^* \, d W = \lb \zeta_t, x^*\rb,  \ \ \ t\in \R_+,\]
then $\phi$ represents an element in $L^0_{\F}(\Om;\g(L^2(\R_+;H),X))$, and almost surely one has
\begin{equation}\label{eq:zetastrong}
\int_0^t \phi \, d W = \zeta_t,  \ \ \ t\in \R_+.
\end{equation}
Moreover, $\zeta$ is a local martingale with continuous paths almost surely.
\end{thm}
This theorem is contrasted by the following
example \cite[Theorem 2.1]{OndrVer2013}.

\begin{ex}
If $X$ is an infinite-dimensional Hilbert space, then there exists a
strongly measurable adapted process $\phi:(0,1)\times\Om\to X$ with the following properties:
\begin{enumerate}[\rm (i)]
 \item for all $x\in X$, the real-valued process $[\phi, x]$ belongs to $L^0(\Om;L^2(0,1))$ and we have
\begin{equation*}
\int_0^1 [ \phi, x] \, d W = 0, \ \text{almost surely};
\end{equation*}
 \item $\|\phi\|_{L^2(0,1;X)} = \infty$ almost surely.
\end{enumerate}
In particular, $\phi$ does not define an element of $L^0(\Om;L^2(0,1;X))$.
\end{ex}
Concerning the necessity of the UMD condition we have the following result due to  Garling \cite{Gar}.
Suppose that for a Banach space $X$ and an exponent $p\in (1, \infty)$
the estimates of Theorem \ref{thm:Itoisomorph} hold for all adapted
elementary processes $\phi: \R_+\times\Om\to X$ (we take $H = \R$):
\begin{equation}\label{eq:dec2est}
\frac1{c_p}\|\phi\|_{L^p(\Omega;\gamma^p(L^2(\R_+),X))}  \leq \Big\| \int_0^\infty \phi\,dB \Big\|_{L^p(\Omega;X)}\leq
C_p \|\phi\|_{L^p(\Omega;\gamma^p(L^2(\R_+),X))}.
\end{equation}
Then $X$ is a UMD space, with constant $\beta_{p,X} \leq c_p C_p$.
This result shows that the scope of
Theorem \ref{thm:Itoisomorph} is naturally restricted to the class of UMD spaces.

\begin{rem}\label{rem:UMD-}
In \cite{CoxGei, CoxVer2, Ga2}
the class of Banach spaces in which the right-hand side inequality of \eqref{eq:dec2est}
holds for all adapted elementary processes $\phi$ is investigated.
This class includes all UMD spaces, but also some non-UMD spaces such as the spaces $L^1(\mu)$.
By extrapolation techniques from \cite{Ge97}, (see \cite[Remark 3.2]{CoxVer2}) this implies that
for all $1\leq p\leq q<\infty$,
\[\Big\| \int_0^\cdot \phi\,dW \Big\|_{L^q(\Omega;C_{\rm b}(\R_+;X))} \leq
 C_{X,p} q \|\phi\|_{L^q(\Omega;\gamma^p(L^2(\R_+;H),X))}.\]
This shows that an estimate with linear dependence in $q$ holds.
\end{rem}

\begin{rem}
In the case when $X$ is a Hilbert space or $X = L^p(\mu)$ with $p\geq 1$, it is known that
\[\Big\|\int_0^\cdot \phi\,dW \Big\|_{L^p(\Omega;C_{\rm b}(\R_+;X))}
\leq  C_{p,X} \|\phi\|_{L^p(\Omega;\gamma^p(L^2(\R_+;H),X))}\]
holds with a constant $C_{p,X}\leq C_X$. In particular,
\[\Big\|\int_0^\cdot \phi\,dW \Big\|_{L^p(\Omega;C_{\rm b}(\R_+;X))}
\leq  C_{p,X} \|\phi\|_{L^p(\Omega;\gamma(L^2(\R_+;H),X))}\]
holds with a constant $C_{p,X}\leq C_X' \sqrt{p}$
(for Hilbert spaces this also follows from Seidler's result quoted earlier,
and for $L^p$ from Fubini's theorem).
It would be interesting to know whether this remains true for arbitrary (UMD) Banach spaces $X$.
This  problem is open even in the case $X = L^q$ with $q\in (1, \infty)\setminus\{2,p\}$.
\end{rem}

We continue with a version of It\^o's lemma taken from \cite{BNVW}.
Let $X,Y,Z$ be Banach spaces and let $(h_n)_{n\geq 1}$ be an orthonormal basis
of $H$. Let $R\in \g(H,X)$, $S\in \g(H,Y)$ and $T\in \calL(X,\calL(Y,Z))$ be
given. It is not hard to show that the sum
\begin{equation}\label{trdef}
{\rm tr}_{R,S} T:=\sum_{n\geq 1} (T R h_n)(S h_n)
\end{equation}
converges in $Z$ and does not depend on the choice of the orthonormal basis. Moreover,
\begin{equation}\label{trineq}
\|{\rm tr}_{R,S} T\|\leq \|T\| \|R\|_{\g(H,X)} \|S\|_{\g(H,Y)}.
\end{equation}
If $X=Y$ we shall write ${\rm tr}_R := {\rm tr}_{R,R}$.

\begin{prop}[It\^o lemma]
Let $X$ and $Y$ be UMD spaces. Assume that $f:\R_+\times X\to Y$ is
of class $C^{1,2}$ on every bounded interval. Let $\phi :\R_+\times \Omega \to \calL(H,X)$ be
$H$-strongly measurable and adapted and assume that $\phi$ locally
defines an element of $L^0(\Omega; \gamma(L^2(\R_+;H),X))\cap L^0(\O;L^2(\R_+;\g(H,X)))$.
 Let $\psi:\R_+\times\O\to X$ be strongly measurable and adapted with locally integrable paths
almost surely. Let $\xi:\O\to X$ be strongly
$\F_0$-measurable.  Define $\zeta:\R_+\times\O\to X$ by
$$
\zeta=\xi + \int_0^\cdot \psi_s \, ds+ \int_0^\cdot \phi_s \, dW_s.
$$
Then $s\mapsto D_2 f(s, \zeta_s) \phi_s$ is stochastically integrable and
almost surely we have, for all $t\ge 0$,
\begin{equation}
\begin{aligned} f(t, \zeta_t)-f(0, \zeta_0) = & \int_0^t D_1 f(s, \zeta_s) \, ds + \int_0^t D_2
f(s, \zeta_s) \psi_s\, ds
\\ & \qquad + \int_0^t D_2 f(s, \zeta_s) \phi_s\, dW_s
 + \frac{1}{2}\int_0^t {\rm tr}_{\phi_s} \big(D_2^2 f(s, \zeta_s) \big) \,ds.
\end{aligned}\label{eq:itoformula}
\end{equation}
\end{prop}
The first two integrals and the last integral are almost surely defined as a
Bochner integral.

As a special case, let $X$ be a UMD space, let $X_1 = X$, $X_2 = X\s$, and set
$$(\zeta_i)_t =\xi_i + \int_0^t (\psi_i)_s \, ds+ \int_0^t (\phi_i)_s \, dW_s, \quad i=1,2,
$$
where $\phi_i: \R_+\times\O\to\calL(H,X_i)$,
$\psi_i: \R_+\times\O\to X_i$ and $\xi_i:\O\to X_i$
satisfy the assumptions of It\^o's lemma.
Then, almost surely, for all $t\ge 0$ we have
\begin{equation}\label{itodual}
\begin{aligned}
\lb (\zeta_1)_t, (\zeta_2)_t\rb - \lb (\zeta_1)_0, (\zeta_2)_0\rb = & \int_0^t \lb
(\zeta_1)_s, (\psi_2)_s\rb + \lb (\psi_1)_s, (\zeta_2)_s\rb \, ds
\\ & \qquad + \int_0^t \lb (\zeta_1)_s, (\phi_2)_s\rb + \lb (\phi_1)_s, (\zeta_2)_s\rb\, dW_s
\\ & \qquad + \int_0^t \sum_{n\geq 1} \lb (\phi_1)_s h_n, (\phi_2)_s h_n\rb \,ds.
\end{aligned}
\end{equation}

\section{Malliavin calculus}

The techniques of the previous section lend themselves very naturally to set
up a Malliavin calculus in UMD Banach spaces.

Let $\H$ be a Hilbert space and let  $W: \H\to L^2(\Om)$
be an isonormal Gaussian process (cf. Definition \ref{def:isonormal}).
The {\em Malliavin derivative}  of an $X$-valued smooth random variable
of the form
  $$F = f(Wh_1,\ldots,Wh_n)\otimes x$$ with
 $f\in C_{\rm b}^{\infty}(\R^n)$, $h_1,\ldots,h_n\in \H$ and $x\in X$,
 is the random variable $DF:\Om\to \g(\H,X)$ defined by
$$
 DF = \sum_{j=1}^n \partial_j f (Wh_1, \ldots, Wh_n)\otimes (h_j \otimes x).
$$
Here, $\partial_j$ denotes the $j$-th partial derivative.
The definition extends by linearity.
Thanks to the integration by parts formula
$$\E\lb DF(h), G\rb) = \E(Wh \lb F,G\rb) - \E \lb F, DG(h)\rb,$$
valid for smooth random variables $F$ and $G$ with values in $X$ and $X\s$, respectively,
the operator $D$ is closable as a densely defined linear operator
from $L^p(\Om;X)$ into $L^p(\Om;\gamma(\H,X))$, $1\le p<\infty$
(see \cite[Proposition 3.3]{MN08}).
The domain of its closure in $L^p(\Om;X)$ is denoted by
$\dD^{1,p}(\Om;X).$ This is a Banach space endowed with the norm
 \begin{align*}
 \|F\|_{\dD^{1,p}(\Om;X)} := ( \|F\|_{L^p(\Om;X)}^p
           + \|DF\|_{L^p(\Om;\g(\H,X))}^p  )^{1/p}.
 \end{align*}

Let $(H_m)_{m\ge 0}$ denote the
Hermite polynomials, given by $H_0(x) = 1$, $H_1(x) = x$, and the recurrence relation
$(m+1)H_{m+1}(x) = x H_m(x) - H_{m-1}(x)$. Let
$$\H_m = \overline{\rm{lin}} \{H_m(Wh): \ \n h\n=1\}, \quad m\ge 0.$$
The Wiener-It\^o decomposition theorem asserts that
$$ L^2(\Om,\cG) = \bigoplus_{m\ge 0} \H_m,$$
where $\cG$ is the $\sigma$-algebra generated by $W$.
Let $P$ be the Ornstein-Uhlenbeck semigroup on $L^2(\Om,\cG)$,
$$ P(t) := \sum_{m\ge 0} e^{-mt} J_m,$$
where $J_m$ is the orthogonal projection onto $\H_m$.
The semigroup $P\otimes I_X$ extends to a strongly continuous semigroup of contractions
on $L^2(\Om,\G;X)$. Its generator will be denoted by $L_X$.

The following result is due Pisier \cite{Pis88}.

 \begin{thm}[Meyer inequalities] \label{thm:meyergeneral}
Let $X$ be a UMD space and let $1 < p < \infty$.
Then $$\Dom_p((-L_X)^{1/2}) = \dD^{1,p}(\O;X)$$ and  for all $F\in\dD^{1,p}(\O;X)$ we have
an equivalence of the homogeneous norms
\begin{align*}
  \|D F\|_{L^p(\O;\gamma(\H,X))}
     &\eqsim_{p,X} \|(L\otimes I_X)^{1/2} F \|_{L^p(\O;X)}.
 \end{align*}
 \end{thm}
An extension to higher order derivatives was obtained by Maas \cite{Maa10}.
We refer the reader to this paper for more on history of vector-valued Malliavin calculus.

From now on we assume that $X$ is
a UMD space. Since UMD spaces are $K$-convex, trace duality establishes a canonical isomorphism
$$\gamma(\H,X\s) \eqsim (\gamma(\H,X))\s.$$
See \cite{KaWe, PisConv} for a proof. We apply this with $X$ replaced by $X\s$
and note that $X$, being a UMD space, is $K$-convex.
Starting from the Malliavin derivative $D$ on $L^{p'}(\Om;X\s)$ with $1<p<\infty$ and $\frac1p+\frac1{p'}=1$,
we define the {\em Skorohod integral} $\delta$ as the adjoint of $D$; thus, $\delta$
is a densely defined closed linear operator from $L^p(\Om;\gamma(\H,X)$ into $L^p(\Om;X)$, $1<p<\infty$.
The domain of its closure will be denoted by $\Dom_p(\delta)$.

So far, $\H$ has been an arbitrary Hilbert space. We now specialise to $\H = L^2(\R_+;H)$
and let $(\F_t)_{t\ge 0}$ be the filtration induced by $W$ (see Section \ref{subsec:random}).
The following result has been proved in \cite{MN08}:

\begin{thm} \label{thm:skorohodintegral}
Let $X$ be a UMD space and let $1<p<\infty$ be given. The space
$L^p_{{{\F}}}(\Om;\g(L^2(\R_+;H),X))$ is contained  in
$\Dom_p(\d)$ and
$$\d(\phi) = \int_0^\infty \phi\,dW, \quad
 \phi \in L^p_{{{\F}}}(\Om;\g(L^2(\R_+;H),X)).$$
\end{thm}
Let $\F^W$ denote the filtration generated by $W$ and define
step functions $f: \R_+\to \gamma(H,L^p(\Om;X))$ with bounded support,
$$
 (P_{\F^W} f)(t) :=\E(f(t) | \F_t^W),
$$
where $\E(\cdot|\F_t^W)$ is considered as a bounded operator acting
on $\gamma(H,L^p(\Om;X))$.
It is shown in \cite{MN08} that if  $1 < p,q < \infty$ satisfy
$\frac1p+\frac1q = 1$, then the mapping
$ P_{\F^W}$ extends to a bounded operator on $\gamma(L^2(\R_+;H),L^p(\Om;X))$.
Moreover, as a bounded operator on $L^p(\Om;\g(L^2(\R_+;H),X))$,
$P_{\F^W}$ is a projection onto the closed subspace
$L_{\F^W}^p(\Om;\gamma(L^2(\R_+;H),X)).$

\begin{thm}[Clark-Ocone representation, \cite{MN08}]\label{thm:ClarkOcone}
Let $X$ be a UMD space.
The operator
$P_{\F^W}\circ D$ has a unique extension to a continuous
operator from
$L^1(\Om,\F_\infty^W;X)$ to $L_{\F^W}^0(\Om;\gamma(L^2(\R_+;H),X))$, and for all
$F\in L^1(\Om,\F_\infty^W;X)$ we have the representation
 \begin{align*}
 F = \E(F) + I ((P_{\F^W} \circ D) F),
 \end{align*}
where $I$ is the stochastic integral with respect to $W$.
Moreover,  $(P_{\F^W} \circ D) F$ is the unique element $\phi\in
L_{\F^W}^0(\Om;\g(L^2(\R_+;H),X))$ satisfying $F = \E(F) + I(\phi)$.
 \end{thm}
 The UMD Malliavin calculus has been
pushed further in the recent paper \cite{ProVer},
where in particular the authors obtained an It\^o formula for the Skorohod integral.

\section{Stochastic maximal $L^p$-regularity}

Applications of the theory of stochastic integration in UMD spaces have been worked
out in a number of papers; see
\cite{BNVW, BrzVer11, CKLL, CoxGor, CoxNee12, Cre, DesLon, KunzeMart, Kunze12, KuNe12, NVW3,
NVW12eq,  NVW10, ProVer, SchnVer10}
and the references therein. Here we will limit ourselves to the maximal regularity theorem for stochastic convolutions
from \cite{NVW10} which is obtained by combining Theorem \ref{thm:main} and \ref{thm:J-Lq} below,
and which
crucially depends on the sharp two-sided inequality of Theorem \ref{thm:Itoisomorph}.

As before we let $(\Omega,\P)$ be a probability space,
let $W$ be an $H$-cylindrical Brownian motion defined on it,
and let the filtration $\F$ be as before.
For an operator $A$ admitting a bounded $H^\infty$-calculus,
we denote by $(S(t))_{t\ge 0}$ the bounded analytic semigroup generated by $-A$.
For detailed treatments of the $H^\infty$-calculus we refer to \cite{DHP, Haase:2, KuWe}.

The main result of \cite{NVW10} is formulated for $L^q$-spaces with $q\in [2,\infty)$,
but inspection of the proof shows that can be restated for UMD spaces satisfying
a certain hypothesis which will be explained in detail below.

\begin{thm}\label{thm:main}
Let $p\in [2, \infty)$ and let $X$ be a UMD Banach space with type $2$
which satisfies Hypothesis $(H_p)$.
Suppose the operator $A$ admits a
bounded $H^\infty$-calculus of angle less than $\pi/2$ on $X$ and let
$(S(t))_{t\ge 0}$ denote the bounded analytic semigroup
on $X$ generated by $-A$.
For all $G\in L_\F^p(\R_+\times\Omega;\g(H,X))$
the stochastic convolution process
\begin{equation}\label{eq:USGdW}
U(t) = \int_0^t S(t-s)G_s\,dW_s, \quad t\ge 0,
\end{equation}
is well defined in $X$, takes values in the fractional domain
$\Dom(A^{1/2})$ almost surely, and we have the stochastic maximal $L^p$-regularity
estimate
\begin{equation}\label{eq:Apq}
\E \n A^{1/2} U\n_{L^p(\R_+;X)}^p \le C^p \E\n
G\n_{L^p(\R_+;\g(H,X))}^p
\end{equation}
with a constant $C$ independent of $G$. If, in addition to the above assumptions, we have
$0\in\varrho(A)$, then \begin{align}\label{eq:trace}
\E \|U\|_{BUC(\R_+;
(L^q(\mathcal{O}),\Dom(A))_{\frac12-\frac1p,p})}^p \leq
C^p \, \E\|G\|_{L^p(\R_+;L^q(\mathcal{O};H))}^p.
\end{align}
\end{thm}
In the special case of $X = L^q(\mathcal{O})$, where $(\mathcal{O},\mu)$ is
a $\sigma$-finite measure space and $q\in [2, \infty)$,
Hypothesis $(H_p)$ is fulfilled for all $p\in (2, \infty)$; the value $p=2$ is allowed if $q=2$
(see Theorem \ref{thm:J-Lq} below).
In this special case, \eqref{eq:Apq} is equivalent to the estimate
\begin{equation}\label{eq:Apq2}
\E \n A^{1/2} U\n_{L^p(\R_+;L^q(\mathcal{O}))}^p \le C^p \E\n
G\n_{L^p(\R_+;L^q(\mathcal{O};H))}^p.
\end{equation}

The convolution process $U$ defined by \eqref{eq:USGdW}
is the mild solution of the abstract SPDE
$$ dU(t) + AU(t)\,dt = G_t\,dW_t, \quad t\ge 0,$$
and therefore Theorem \ref{thm:main}
can be interpreted as a maximal $L^p$-regularity result for such equations.
As is well-known \cite{Brz1, DPZ, Kry}, stochastic maximal regularity estimates can be combined
with fixed point arguments to
obtain existence, uniqueness and regularity results for solutions to
more general classes of nonlinear stochastic PDEs. For the setting considered here
this has been worked out in detail in \cite{NVW12eq}, where an application
is included for Navier-Stokes equation with multiplicative gradient-type noise.

Theorem \ref{thm:main} generalises previous results due to Krylov \cite{Kry94, Kry, Kry00, Kry06}
who proved the estimate for second-order uniformly elliptic operators on $X = L^q(\R^d)$
with $2\leq q\leq p$,
where $D=\R^d$ or $D$ is a smooth enough bounded domain in $\R^d$.
Using PDE arguments, Krylov was able to prove his result for operators with coefficients
which may be both time-dependent and random in an adapted and measurable way. These results
were extended to half-spaces and bounded domains by Kim \cite{Kim}.

The proof of Theorem \ref{thm:main} for $X = L^q(\mathcal{O})$ in \cite{NVW10} consists of three main steps:
\begin{enumerate}[\rm(i)]
 \item The $H^\infty$-calculus of
$A$ is used to obtain a reduction to an estimate
for stochastic convolutions of scalar-valued kernels;
 \item This estimate is then proved using Hypothesis $(H_p)$.
 \item Hypothesis $(H_p)$ is verified for $X = L^q(\mathcal{O})$.
\end{enumerate}
In this section we shall present a proof of Theorem \ref{thm:main} which replaces (i) and (ii)
by a simpler $H^\infty$-functional calculus argument.

Let us first turn to the precise formulation of Hypothesis $(H_p$). Let $\mathcal{K}$
be the set of all absolutely continuous functions $k:\R_+\to \R$ such
that $\lim_{t\to \infty} k(t) = 0$ and $$\int_0^\infty t^{1/2} | k'(t)| \, dt \leq 1.$$
Fix
$p\in [2,\infty)$ and let $X$ be an arbitrary Banach space.
For $k\in \mathcal{K}$ and adapted elementary processes
$G:\R_+\times\O\to \calL(H,X)$ we define the process
$I(k) G:\R_+\times\O\to X$ by
\begin{align}\label{eq:Ikoperator}
(I(k) G)_t := \int_0^t k(t-s) G_s \, dW_s, \ \ t\ge 0.
\end{align}
Since $G$ is an adapted elementary process, the It\^o isometry for
scalar-valued processes shows that these stochastic integrals are well-defined
for all $t\ge 0$; no condition on $X$ is needed for this. If $X$ has martingale type $2$
(in particular, when $X$ is UMD with type $2$),
then by Theorem \ref{thm:Mtype2-stochint} and Young's inequality it is easy see
that $I(k)$ extends to a bounded operator from $
L^{p}_\mathscr{F}(\R_+\times\O;\g(H,X))$ into
$L^{p}(\R_+\times\O;X)$ and that the family
\[\mathcal{I} := \{I(k): k\in \mathcal{K}\}\]
is uniformly bounded.
We will need that this family has the stronger property of being $R$-bounded.

A family $\mathscr{T}$ of bounded linear operators from a Banach space $X_1$ into
another Banach space $X_2$
is called {\em $R$-bounded} if there exists a constant $C\ge 0$ such
that for all finite sequences $(x_n)_{n=1}^N$ in $X_1$ and
$(T_n)_{n=1}^N$ in ${\mathscr {T}}$ we have
\[ \E \Big\n \sum_{n=1}^N r_n T_n x_n\Big\n^2
\le C^2\E \Big\n \sum_{n=1}^N r_n x_n\Big\n^2.
\]
Every $R$-bounded family is uniformly bounded; the converse
holds if (and only if, see \cite{AreBu02}) $X$ has cotype $2$ and $Y$ has type $2$.
In particular, the
converse holds if $X_1$ and $X_2$ are Hilbert spaces. The notion of $R$-boundedness
has been first studied systematically in
\cite{CPSW}; for further results and historical remarks see \cite{DHP, KuWe}.

Now we are ready to formulate Hypothesis $(H_p)$:

\begin{enumerate}
\item[$(H_p)$] Each of the operators $I(k)$, $k\in \mathcal{K}\}$, extends to a bounded operator
from $L^{p}_\mathscr{F}(\R_+\times\O;\g(H,X))$ into $L^{p}(\R_+\times\O;X)$, and
the family $$\mathcal{I} = \{I(k): k\in \mathcal{K}\}$$ is $R$-bounded
from $L^{p}_\mathscr{F}(\R_+\times\O;\g(H,X))$ into $L^{p}(\R_+\times\O;X)$.
\end{enumerate}
One can show that if the operators $I(k)$ extend to a uniformly bounded family of bounded operators
from $L^{p}_\mathscr{F}(\R_+\times\O;\g(H,X))$ into $L^{p}(\R_+\times\O;X)$,
then $p\geq 2$ and $X$ has type $2$ (see \cite{NVW10b}).
If $X$ satisfies $(H_p)$ and $Y$ is isomorphic to a
closed subspace of $X$, then $Y$ satisfies $(H_p)$ as well.

Hypothesis $(H_p)$ admits various equivalent formulations. We present one of them, implicit in \cite{NVW10};
for a systematic study we refer the reader to
\cite{NVW10b}. Let $B$ be a real-valued Brownian motion.
\begin{prop}
Hypothesis $(H_p)$ holds if and only if the family $\{I_t: t>0\}$ of stochastic convolution operators defined by
$$I_t g(s):= \int_0^t \frac{1}{\sqrt{t}}\one_{(0,t)}(s-r)g(r)\,dB_r, \quad s\ge 0,
$$
is well defined and $R$-bounded from
$L^{p}(\R_+;X)$ into $L^{p}(\R_+\times\O;X)$.
\end{prop}
Stated differently, in order to verify $(H_p)$ it suffices to take $H=\R$ and to consider the kernels
$\frac{1}{\sqrt{t}}\one_{(0,t)}$, $t>0$.

The following theorem gives sufficient conditions for $(H_p)$ in case $X = L^q(\mathcal{O})$.
\begin{thm} \label{thm:J-Lq}
Let $X$ be isomorphic to a closed subspace of a space $L^q(\mathcal{O})$ with $q\in [2, \infty)$. Then
$(H_p)$ holds for all $p\in (2,\infty)$.
The same result holds when $p=q=2$.
\end{thm}
This is a non-trivial result which has been proved in \cite{NVW10} using the
Fefferman--Stein maximal theorem; it is here that the full force of Theorem \ref{thm:Itoisomorph}
is needed.  By the above remarks, $(H_p)$ also holds for Sobolev spaces $W^{\alpha,p}(\mathcal{O})$
as long as $p\in [2, \infty)$.
It is an open problem to describe the class of
Banach spaces $X$ to which
the result of Theorem \ref{thm:J-Lq} can be extended. A sufficient condition
for Hypothesis $(H_p)$ for any $p>2$ is that $X$ be a UMD
Banach function space for which the norm can be
written as $\|x\|_X  = \| \, |x|^2\, \|_F$, and $F$ is another UMD Banach
function space \cite{NVW10b}.

In order to set the stage for the proof of Theorem \ref{thm:main} we need to introduce some terminology.
Let $X$ be a
Banach space
and let $\Sigma_\sigma = \{z\in \C\setminus\{0\}: \ |\arg z| < \sigma\}$ denote the open sector of angle $\sigma$
about the positive real axis in the complex plane.
Let $A$ be a sectorial operator on $X$ with a bounded $H^\infty(\Sigma_\sigma)$-calculus.
Following \cite{KWcalc} and \cite[Chapter 12]{KuWe}, we
denote by $\A$ the sub-algebra of $\calL(X)$ of all operators commuting with the
resolvent $R(\lambda,A) = (\lambda-A)^{-1}$. For $\nu>\sigma$,
the space of all bounded analytic functions
$f: \Sigma_\nu\to \A$ with $R$-bounded range is denoted by
$RH^\infty(\Sigma_\nu,\A)$. By $RH^\infty_0(\Sigma_\nu,\A) $ we denote
the functions in $RH^\infty(\Sigma_\nu,\A)$ whose operator norm is dominated by
$|\lambda|^\eps/(1+|\lambda|)^{2\eps}$ for some $\eps>0$. For such $f$ we may define
$$ f(A) = \frac1{2\pi i} \int_{\partial \Sigma_{\sigma'}} f(\lambda) R(\lambda,A)\,d\lambda$$
as an absolutely convergent Bochner integral in $\calL(X)$ for $\sigma<\sigma'<\nu$.
By \cite[Theorem 4.4]{KWcalc} (see also \cite[Theorem 12.7]{KuWe}), the mapping $f\mapsto f(A)$ extends
to a bounded algebra homomorphism from $RH^\infty(\Sigma_\nu,\A)$
to $\calL(X)$ which is unique in the sense that it has the following
convergence property: if $(f_n)$ is a bounded sequence in $RH^\infty(\Sigma_\nu,\A)$
(in the sense that the corresponding $R$-bounds are uniformly bounded) and
$f_n(\lambda)x \to f(\lambda)x$ for some $f\in RH^\infty(\Sigma_\nu,\A)$ and all $\lambda\in \Sigma_\nu$
and $x\in X$, then $f_n(A)x \to f(A)x$ for all $x\in X$.

\begin{proof}[Proof of Theorem \ref{thm:main}]
Let $X$ be a UMD Banach space with type $2$
satisfying Hypothesis $(H_p)$ for some fixed $p\in [2, \infty)$.
For adapted elementary processes $G:\R_+\times\O\to \g(H,X)$
and $\lambda\in\C$ with $\Re\lambda>0$
define
$$(L_\lambda G)_t := \int_0^t \lambda^{1/2} e^{-\lambda(t-s)}\,G_s\,dW_s, \quad t>0.$$
The functions $k_\lambda (t):=  \lambda^{1/2} e^{-\lambda t}$ are uniformly bounded in the norm of $L^2(\R_+)$
and therefore Young's inequality and Theorem \ref{thm:Burk-Mtype2}
show that the operators
$L_\lambda$ are bounded from $L_\F^p(\R_+\times \O;\g(H,X))$ to  $L^p(\R_+\times \O;X)$.
Moreover, the substitution $t\Re\lambda = s$ gives, for $\lambda\in\Sigma_\nu$,
\begin{align*}\int_0^\infty t^{1/2} | k_\lambda'(t)| \, dt
\le \frac1{\sqrt{\cos \nu}} \int_0^\infty  s^{1/2} e^{-s}\,ds =  \frac12\sqrt{\frac{\pi}{\cos \nu}}.
\end{align*}
This shows that the functions $k_\lambda$, $\lambda\in\Sigma_\nu$,
belong to $\mathcal{K}$ after scaling by a constant depending only on $\nu$.
Hence, by Hypothesis $(H_p)$, for any $0\le \nu<\frac12\pi$ the family
$\{L_\lambda: \ \lambda\in \Sigma_\nu\}$ is $R$-bounded from
$L_\F^p(\R_+\times \O;\g(H,X))$ to  $L_\F^p(\R\times \O;X)$.

In order to view the operators $L_\lambda$ as bounded operators
on $\wt X:= L_{\F}^p(\R_+\times\O;\gamma(H,X))$ we think of $X$ as being embedded
isometrically as a closed subspace of $\gamma(H,X)$ by identifying
each $x\in X$ with the rank one operator $h_0\otimes x$,
where $h_0\in H$ is an arbitrary but fixed unit vector. Using this identification,
$L_{\F}^p(\R_+\times\O;X)$ is isometric to a closed subspace of $\wt X$ and we may identify
$L_\lambda$ with a bounded operator $\wt L_\lambda$ on $\wt X$; the resulting family $\{\wt L_\lambda: \ \lambda\in \Sigma_\nu\}$
is $R$-bounded on $\wt X$.

Suppose $A$ has a bounded $H^\infty(\Sigma_\sigma)$-calculus on $X$
for some $\sigma\in [0,\frac12\pi)$.
Let $\wt A$ denote the induced operator on
$\wt X = L_{\F}^p(\R_+\times\O;\g(H,X))$,
given by
$(\wt A G)_t := A (G_t)$ for $G\in L_{\F}^p(\R_+\times\O;\g(H,\Dom(A)))$.
It is routine to check that $\wt A$ has a bounded $H^\infty(\Sigma_\sigma)$-calculus on $\wt X$
and
$$ (\varphi(\wt A) G)_t = \varphi(A) (G_t) .$$
Noting that the operators $\wt L_\lambda$ and $R(\lambda,\wt A)$ commute,
the above-mentioned result from \cite{KWcalc}, applied to the function
$$ f(\lambda) =\wt L_\lambda, \quad \lambda\in \Sigma_\nu,$$ shows that the operator
$$G \mapsto
f(\wt A)G = \int_{\partial \Sigma_{\sigma'}} R(\lambda,\wt A) \wt L_\lambda G\,d\lambda
$$ with $\sigma<\sigma'<\nu$, is well defined and bounded on $\wt X$.
It follows that the operator
$$G \mapsto
f(\wt A)G = \int_{\partial \Sigma_{\sigma'}} R(\lambda,\wt A) L_\lambda G\,d\lambda
$$ with $\sigma<\sigma'<\nu$, is well defined and bounded from $\wt X$ to $L_{\F}^p(\R_+\times\O;X)$
(cf. \cite[Theorem 4.5]{KWcalc}).
By the stochastic Fubini theorem, for adapted elementary processes
$G: \R_+\times\O\to \calL(H,\Dom(A))$ we have, for all $t>0$,
\begin{align*}
(f(\wt A)G)_t
& = \int_{\partial \Sigma_{\sigma'}} R(\lambda,\wt A) L_\lambda G_t\,d\lambda
\\ & = \int_{\partial \Sigma_{\sigma'}} \int_0^t \lambda^{1/2} e^{-\lambda (t-s)} R(\lambda, A)G_s\,dW_s\,d\lambda
\\ & =  \int_0^t\int_{\partial \Sigma_{\sigma'}} \lambda^{1/2} e^{-\lambda (t-s)} R(\lambda, A)G_s\,d\lambda \,dW_s
\\ &  = \int_0^t  A^{1/2} e^{-(t-s)A}G_s\,dW_s.
\end{align*}
Putting the results together we obtain
\begin{align*}
\Big\n t\mapsto \int_0^t  A^{1/2} S(t-s)G_s\,dW_s\Big\n_{L^p(\R_+\times\Om;X)}
& = \n f(\wt A)G\n_{L^p(\R_+\times\Om;X)}
\\ & \le C^p \n G\n_{L^p(\R_+\times\Om;\gamma(H,X))}
\end{align*}
This proves Theorem \ref{thm:main}.
\end{proof}

Next we deduce a variant of
Theorem \ref{thm:main} for processes with mixed integrability assumptions.
Its proof is a straightforward application of the two-sided estimates for stochastic integrals in UMD spaces.

\begin{cor}
Let the assumptions of Theorem \ref{thm:main} be satisfied,
and
let $G\in L^r_{\F}(\O;L^p(\R_+;\g(H,X)))$ with $r\in (0,\infty)$ be given.
If $U$ is defined as in \eqref{eq:USGdW}, then
\begin{equation}\label{eq:Apq3}
\E \n A^{1/2} U\n_{L^p(\R_+;X)}^r \le C^r \E\n
G\n_{L^p(\R_+;\g(H,X))}^r
\end{equation}
with a constant $C$ independent of $G$.
\end{cor}
\begin{proof}
By Proposition \ref{prop:Ito} and Theorem \ref{thm:main},
applied to deterministic functions $G\in L^p(\R_+;\g(H,\Dom(A)))$, we have
\begin{equation}\label{eq:detG}
\|s\mapsto A^{1/2}S(t-s) \one_{[0,t](s)} G_s\|_{\g(L^2(\R_+;H), L^p(\R_+;X))} \leq C \|G\|_{L^p(\R_+;\g(H,\Dom(A)))}.
\end{equation}

Next let $G\in L^r_{\F}(\O;L^p(\R_+;\g(H,\Dom(A))))$.
By Theorem \ref{thm:Itoisomorph} (or rather, by its extension to the closure of the elementary
adapted processes, cf. \eqref{eq:I-isom})
applied to the UMD space $L^p(\R_+;X)$  we obtain
\begin{align*}
\ & \n A^{1/2} U\n_{L^r(\O;L^p(\R_+;X))} \\ & \qquad \eqsim
\|s\mapsto A^{1/2}S(t-s) \one_{[0,t](s)} G_s\|_{L^r(\O;\g(L^2(\R_+;H), L^p(\R_+;X)))}.
\end{align*}
Now \eqref{eq:Apq3} follows by applying the estimate \eqref{eq:detG} pointwise in $\O$.
\end{proof}

\begin{rem}
 A variation of the notion of stochastic maximal $L^p$-regularity, in which the $L^p(\R_+;X)$-norm
over the time variable is replaced by the $\gamma(L^2(\R_+),X)$-norm, has been studied in \cite{NVW-gamma}.
With this change, a stochastic maximal $L^p$-regularity result holds for arbitrary UMD Banach spaces
with Pisier's property $(\alpha)$ and all exponents $0<p<\infty$.
In this situation the trace inequality \eqref{eq:trace} holds with $(X,\Dom(A))_{\frac12-\frac1p,p}$ replaced by $X$.
\end{rem}

\section{Poisson stochastic integration}

Up to this point we have been exclusively concerned with the Gaussian case.
Here we shall briefly address the problem of extending Theorem \ref{thm:Itoisomorph} to more general
classes of integrators. More specifically, with an eye towards the L\'evy case, a natural
question is whether similar two-sided estimates as in Theorem \ref{thm:Itoisomorph} can be given
in the Poissonian case.
This question has been addressed recently by Dirksen \cite{Dirksen}, who was able to work out the
correct norms in the special case $X = L^q(\mathcal{O})$.

We begin by recalling some standard definitions.  Let $(\Om,\F,\P)$ be a
probability space and let $(E,\calE)$ be a measurable space.
We write $ \overline \N = \N \cup\{\infty\}$.

\begin{defn}
\label{def:randomMeasure}
\index{random measure!Poisson}
A \emph{random measure} is a mapping
$N: \Omega\times\cE\to \overline \N$
with the following properties:
\begin{enumerate}
\renewcommand{\labelenumi}{\rm(\roman{enumi})}
\item For all $B\in\cE$ the mapping $N(B): \omega\mapsto N(\om,B)$ is measurable;
\item For all $\om\in\Om$, the mapping $B\mapsto N(\om,B)$ is a measure.
\end{enumerate}
The measure $\mu(B):= \E N(B)$ is called the \emph{intensity measure} of $N$. \par
\end{defn}

\begin{defn}
 A random measure $N : \Omega\times\cE\to \overline \N $
with intensity $\mu$ is called a {\em Poisson random measure} if the following conditions are satisfied:
\begin{enumerate}
\renewcommand{\labelenumi}{\rm(\roman{enumi})}
\addtocounter{enumi}{2}
\item For all pairwise disjoint sets $B_1,\ldots,B_n$ in $\cE$ the random variables
$N(B_1)$, $\ldots$\,, $N(B_n)$ are independent;
\item For all $B \in \cE$ with $\mu(B)<\infty$ the random variable $N(B)$ is
Poisson distributed with parameter $\mu(B)$.
\end{enumerate}
\end{defn}
Recall that a random variable $f:\Om\to \N$ is {\em Poisson distributed} with {\em parameter} $\lambda > 0$
if \begin{align*}
\P(f = n ) =\frac{\lambda^n}{n!} e^{-\lambda}, \qquad n \in \N.
\end{align*}
For $B \in \cE$ with  $\mu(B)<\infty$ we write
$$\tilde{N}(B) := N(B) - \mu(B).$$
It is customary to call $\tilde N$  the \emph{compensated
Poisson random measure} associated with $N$ (even it is not a random
measure in the sense of Definition \ref{def:randomMeasure}, as it is defined
on the sets of finite $\mu$-measure only).

Let $(J,\cJ,\nu)$ be a $\sigma$-finite measure space and let $N$ be a Poisson
random measure on $(\R_+\times J,{\mathscr{B}}(\R_+)\times \cJ,dt\times \nu)$.
Throughout this section we let
$\F$ be the filtration generated by the random variables
$\{\tilde{N}((s,u]\times A) \ : \ 0\leq s<u\leq t, A \in \cJ\}$.

An {\em adapted elementary process}  $\phi:\Om\times \R_+\times J\rightarrow X$
is a linear combination of processes of the form
$\phi =  \one_{F} \one_{(s,t]\times A} \otimes x$,
with $0\le s< t<\infty$, $A\in \cJ$ satisfying $\nu(A_j)<\infty$, $F\in \F_{s}$, and $x\in X$.
For an adapted elementary process $\phi$ and a set $B \in \cJ$
we define the \emph{(compensated) Poisson stochastic integral} by
\begin{equation*}
\int_{\R_+\times B}   \one_{F} \one_{ (s,t]\times A} \otimes x\, d\tilde{N}: =
\one_F \tilde{N}((s,t]\times (A\cap B)) \otimes x
\end{equation*}
and extend this definition by linearity.

The next two theorems, taken from \cite{DMN},  give an upper and lower
bound for the Poisson stochastic integral of an elementary adapted
process in the presence of non-trivial martingale type and finite martingale cotype, respectively.
Theorem \ref{thm:typesBanach} may be regarded as a Poisson analogue
of Theorem \ref{thm:Mtype2-stochint}.

We write $$ \cD_{s,X}^p := L^p(\Om; L^s(\R_+\times J;X)).$$
\begin{thm}
\label{thm:typesBanach} Let $\phi$ be an elementary adapted process with values in a
 Banach space $X$ with martingale type $s\in (1,2]$.
\begin{enumerate}
 \item[\rm(1)]
If $1<s\leq p<\infty$ we have, for all $B \in \cJ$,
\begin{align*}
  \Big(\E\sup_{t\geq 0}\Big\|\int_{[0,t]\times B} \phi \ d\tilde{N}\Big\|^p\Big)^{1/p}
\lesssim_{p,s,X} \n \one_B \phi\n_{ \cD_{s,X}^p \cap \cD_{p,X}^p}.
\end{align*}
 \item[\rm(2)]
If $1\leq p<s$  we have, for all $B \in \cJ$,
\begin{align*}
 \Big(\E\sup_{t\geq 0}\Big\|\int_{[0,t]\times B} \phi \ d\tilde{N}\Big\|^p\Big)^{1/p}
\lesssim_{p,s,X} \n \one_B \phi\n_{ \cD_{s,X}^p + \cD_{p,X}^p}.
\end{align*}
\end{enumerate}
\end{thm}

Theorem~\ref{thm:typesBanach} extends several known vector-valued inequalities
in the literature. In the special case where $X=\R^n$ and $2\leq p<\infty$, the estimate
(1) was obtained in  \cite[p. 335, Corollary 2.12]{Ku} by a completely different argument based on
It\^{o}\textquoteright s formula. An estimate for Hilbert spaces $X$ and $2\leq p<\infty$ was
obtained in \cite[Lemma 3.1]{MPR}. The estimate (1) is slightly stronger in
this case.
In \cite[Lemma 4]{MR}, a slightly weaker inequality than
(1) was obtained in the special case $X=L^s(\mu)$ and $p=s\geq 2$. This result was deduced
from the corresponding scalar-valued inequality via Fubini's theorem. Finally, in \cite{Hau11},
the inequality (1) was obtained in the special case when $p=s^n$ for some integer $n\ge 1$.
Using a different approach, Theorem~\ref{thm:typesBanach} has been obtained independently
by Zhu \cite{Zhu}.

The following `dual' version of Theorem~\ref{thm:typesBanach} holds for Banach spaces with martingale cotype.
 \begin{thm}
 \label{thm:cotypesBanach} Let $\phi$ be an elementary adapted process with values in a
  Banach space $X$ with martingale cotype $s\in [2,\infty)$.
 \begin{enumerate}
  \item[\rm(1)]
 If $s\leq p<\infty$ we have, for all $B \in \cJ$ and $t\geq 0$,
 \begin{align*}
 \n \one_{[0,t]\times B} \phi\n_{ \cD_{s,X}^p \cap \cD_{p,X}^p}
\lesssim_{p,s,X} \Big(\E \Big\|\int_{[0,t]\times B} \phi \ d\tilde{N}\Big\|^p\Big)^{1/p}.
 \end{align*}
  \item[\rm(2)]
 If $1<p<s$ we have, for all $B \in \cJ$ and $t\geq 0$,
 \begin{align*}
\n \one_{[0,t]\times B} \phi\n_{ \cD_{s,X}^p + \cD_{p,X}^p}
\lesssim_{p,s,X} \Big(\E\Big\|\int_{[0,t]\times B} \phi\ d\tilde{N}\Big\|^p\Big)^{1/p}.
 \end{align*}
 \end{enumerate}
\end{thm}
For Hilbert spaces $X$, Theorems \ref{thm:typesBanach} and \ref{thm:cotypesBanach}
combine to yield two-sided estimates for the $L^p$-norm
of the stochastic integral with respect to a
compensated Poisson random measure.

\begin{cor}
\label{cor:Hilbert} Let $H$ be a Hilbert space and let $\phi$ be an elementary adapted $H$-valued process.
\begin{enumerate}[\rm(1)]\item If $2\leq p<\infty$, then for all $B\in\cJ$  we have
\begin{align*}
\Big(\E \sup_{t\geq 0}\Big\|\int_{[0,t]\times B} \phi \ d\tilde{N}\Big\|^p\Big)^{1/p}
\simeq_p  \n \one_B \phi\n_{ \cD_{s,H}^p \cap \cD_{p,H}^p}.
\end{align*}
\item If $1<p<2$, then for all $B\in\cJ$ we have
\begin{align*}
\Big(\E \sup_{t\geq 0}\Big\|\int_{[0,t]\times B} \phi \ d\tilde{N}\Big\|^p\Big)^{1/p}
\simeq_p  \n \one_B \phi\n_{ \cD_{s,H}^p + \cD_{p,H}^p}.
\end{align*}
\end{enumerate}
\end{cor}

For the spaces $X = L^q(\mathcal{O})$, where $(\mathcal{O},\Sigma,\mu)$ is an arbitrary measure space,
sharp two-sided bounds for the Poisson stochastic integral can be proved. This result, due Dirksen \cite{Dirksen},
may be regarded as the Poisson analogue of Theorem \ref{thm:Itoisomorph} for $X = L^q(\mathcal{O})$.
An alternative proof has been obtained subsequently by Marinelli \cite{marinelli2013maximal}.
We write
\begin{align*}
\cS_q^p := L^p(\Om; L^q(\mathcal{O};L^2(\R_+\times J))), \\
\cD_{s,q}^p := L^p(\Om; L^s(\R_+\times J;L^q(\mathcal{O}))). \\
\end{align*}

\begin{thm}\label{thm:ItoPoisson}
 Let $1<p,q<\infty$. For any $B\in\cJ$ and for any elementary adapted $L^q(\mathcal{O})$-valued process $\phi$,
\begin{equation}
\label{eqn:summarySILqPoisson} \Big(\E \sup_{t\geq 0}
\Big\|\int_{[0,t]\times B} \phi \ d\tilde{N}\Big\|_{L^q(\mathcal{O})}^p\Big)^{1/p} \simeq_{p,q} \|\one_B \phi\|_{\cI_{p,q}},
\end{equation}
where $\cI_{p,q}$ is given by
\begin{align*}
\cS_{q}^p \cap \cD_{q,q}^p \cap \cD_{p,q}^p & \ \ \mathrm{if} \ \ 2\leq q\leq p;\\
\cS_{q}^p \cap (\cD_{q,q}^p + \cD_{p,q}^p) & \ \ \mathrm{if} \ \ 2\leq p\leq q;\\
(\cS_{q}^p \cap \cD_{q,q}^p) + \cD_{p,q}^p & \ \ \mathrm{if} \ \ p\le 2\leq q;\\
(\cS_{q}^p + \cD_{q,q}^p) \cap \cD_{p,q}^p & \ \ \mathrm{if} \ \ q\le 2\leq p;\\
\cS_{q}^p + (\cD_{q,q}^p \cap \cD_{p,q}^p) & \ \ \mathrm{if} \ \ q\leq p\leq 2;\\
\cS_{q}^p + \cD_{q,q}^p + \cD_{p,q}^p & \ \ \mathrm{if} \ \ p\leq q\leq 2.
\end{align*}
\end{thm}
It is also shown that the estimate $\lesssim_{p,q}$ in \eqref{eqn:summarySILqPoisson} remains valid if $q=1$.
A non-commutative version of Theorem \ref{thm:ItoPoisson} in a more general abstract setting can be found in
\cite[Section 7]{Dirksen}.

In contrast to the Gaussian case, where one expression for the norm suffices for all $1<p,q<\infty$,
in the Poisson case
$6$ different expressions are obtained depending on the mutual positions of the numbers $p$, $q$, and $2$.
This also suggests that the problem of determining sharp two-sided bounds for elementary adapted processes with
values in a general UMD space $X$ seems to be a very challenging one.

Noting that $X = L^q(\mathcal{O})$ has martingale type $q\wedge 2$ and martingale cotype $q\vee 2$, Theorems
\ref{thm:typesBanach} and \ref{thm:cotypesBanach} are applicable as well; for $q\not=2$ the bound obtained from these
theorems are weaker that the ones obtained from Theorem \ref{thm:ItoPoisson}.

\subsection*{Acknowledgment} We thank Markus Antoni for carefully reading an earlier draft of this paper.

\def\polhk#1{\setbox0=\hbox{#1}{\ooalign{\hidewidth
  \lower1.5ex\hbox{`}\hidewidth\crcr\unhbox0}}}
  \def\polhk#1{\setbox0=\hbox{#1}{\ooalign{\hidewidth
  \lower1.5ex\hbox{`}\hidewidth\crcr\unhbox0}}} \def\cprime{$'$}


\begin{thebibliography}{100}

\bibitem{AreBu02}
W.~Arendt and S.~Bu.
\newblock The operator-valued {M}arcinkiewicz multiplier theorem and maximal
  regularity.
\newblock {\em Math. Z.}, 240(2):311--343, 2002.

\bibitem{BelDal}
Ya.I. Belopol{\cprime}skaya and Yu.L. Daletski{\u\i}.
\newblock It\^o equations and differential geometry.
\newblock {\em Uspekhi Mat. Nauk}, 37(3(225)):95--142, 224, 1982.

\bibitem{BD}
Ya.I. Belopol{\cprime}skaya and Yu.L. Daletski{\u\i}.
\newblock {\em Stochastic equations and differential geometry}, volume~30 of
  {\em Mathematics and its Applications (Soviet Series)}.
\newblock Kluwer Academic Publishers Group, Dordrecht, 1990.

\bibitem{Bog}
V.I. Bogachev.
\newblock {\em Gaussian measures}, volume~62 of {\em Mathematical Surveys and
  Monographs}.
\newblock American Mathematical Society, Providence, RI, 1998.

\bibitem{Bou83}
J.~Bourgain.
\newblock Some remarks on {B}anach spaces in which martingale difference
  sequences are unconditional.
\newblock {\em Ark. Mat.}, 21(2):163--168, 1983.

\bibitem{Bour86}
J.~Bourgain.
\newblock Vector-valued singular integrals and the {$H\sp 1$}-{BMO} duality.
\newblock In {\em Probability theory and harmonic analysis (Cleveland, Ohio,
  1983)}, volume~98 of {\em Monogr. Textbooks Pure Appl. Math.}, pages 1--19.
  Dekker, New York, 1986.

\bibitem{BroDin}
J.K. Brooks and N~Dinculeanu.
\newblock {\em Stochastic integration in {B}anach spaces}, volume~81 of {\em
  Advances in Mathematics}.
\newblock Academic Press, 1990.

\bibitem{Brz1}
Z.~Brze{\'z}niak.
\newblock Stochastic partial differential equations in {M}-type {$2$} {B}anach
  spaces.
\newblock {\em Potential Anal.}, 4(1):1--45, 1995.

\bibitem{Brz2}
Z.~Brze{\'z}niak.
\newblock On stochastic convolution in {B}anach spaces and applications.
\newblock {\em Stochastics Stochastics Rep.}, 61(3-4):245--295, 1997.

\bibitem{Brz3}
Z.~Brze{\'z}niak.
\newblock Some remarks on {I}t\^o and {S}tratonovich integration in 2-smooth
  {B}anach spaces.
\newblock In {\em Probabilistic methods in fluids}, pages 48--69. World Sci.
  Publishing, River Edge, NJ, 2003.

\bibitem{BrzMil}
Z.~Brze\'zniak and A.~Millet.
\newblock On the stochastic {S}trichartz estimates and the stochastic nonlinear
  {S}chr\"odinger equation on a compact {R}iemannian manifold.
\newblock {\em Potential Analysis}, 2013.
\newblock Online First.

\bibitem{BNVW}
Z.~Brze{\'z}niak, J.M.A.M.~van Neerven, M.C. Veraar, and L.W. Weis.
\newblock It\^o's formula in {UMD} {B}anach spaces and regularity of solutions
  of the {Z}akai equation.
\newblock {\em J. Differential Equations}, 245(1):30--58, 2008.

\bibitem{BrzVer11}
Z.~Brze{\'z}niak and M.C. Veraar.
\newblock Is the stochastic parabolicity condition dependent on {$p$} and
  {$q$}?
\newblock {\em Electron. J. Probab.}, 17(56):1--24, 2012.

\bibitem{Bu1}
D.L. Burkholder.
\newblock A geometrical characterization of {B}anach spaces in which martingale
  difference sequences are unconditional.
\newblock {\em Ann. Probab.}, 9(6):997--1011, 1981.

\bibitem{Bu2}
D.L. Burkholder.
\newblock Martingales and {F}ourier analysis in {B}anach spaces.
\newblock In {\em Probability and analysis (Varenna, 1985)}, volume 1206 of
  {\em Lecture Notes in Math.}, pages 61--108. Springer, Berlin, 1986.

\bibitem{Burk01}
D.L. Burkholder.
\newblock Martingales and singular integrals in {B}anach spaces.
\newblock In {\em Handbook of the geometry of {B}anach spaces, {V}ol. {I}},
  pages 233--269. North-Holland, Amsterdam, 2001.

\bibitem{CM}
A.~Chojnowska-Michalik.
\newblock Stochastic differential equations in {H}ilbert spaces.
\newblock In {\em Probability theory (Papers, VIIth Semester, Stefan Banach
  Internat. Math. Center, Warsaw, 1976)}, volume~5 of {\em Banach Center
  Publ.}, pages 53--74. PWN, Warsaw, 1979.

\bibitem{CKLL}
P.A. Cioica, K.-H. Kim, K.~Lee, and F.~Lindner.
\newblock On the $l^q(l^p)$-regularity and {B}esov smoothness of stochastic
  parabolic equations on bounded {L}ipschitz domains.
\newblock {\em Electron. J. Probab}, 18(82):1--41, 2013.

\bibitem{CPSW}
Ph. Cl{\'e}ment, B.~de Pagter, F.A. Sukochev, and H.~Witvliet.
\newblock Schauder decompositions and multiplier theorems.
\newblock {\em Studia Math.}, 138(2):135--163, 2000.

\bibitem{CoxGei}
S.G. Cox and S.~Geiss.
\newblock On relations between decoupling inequalities in {B}anach spaces.
\newblock in preparation.

\bibitem{CoxGor}
S.G. Cox and M.~G{\'o}rajski.
\newblock Vector-valued stochastic delay equations -- a semigroup approach.
\newblock {\em Semigroup Forum}, 82(3):389--411, 2011.

\bibitem{CoxNee12}
S.G. Cox and J.M.A.M. van Neerven.
\newblock Pathwise {H}\"older convergence of the implicit {E}uler scheme for
  semi-linear {SPDE}s with multiplicative noise.
\newblock {\em Numer. Math.}, 125(2):259--345, 2013.

\bibitem{CoxVer2}
S.G. Cox and M.C. Veraar.
\newblock {Vector-valued decoupling and the {B}urkholder-{D}avis-{G}undy
  inequality.}
\newblock {\em Illinois J. Math.}, 55(1):343--375, 2011.

\bibitem{Cre}
P.~Crewe.
\newblock Infinitely delayed stochastic evolution equations in {UMD} {B}anach
  spaces.
\newblock {\em J. Math. Anal. Appl.}, 415(1):325--345, 2014.

\bibitem{Cur}
R.F. Curtain.
\newblock {\em Stochastic differential equations in a {H}ilbert space}.
\newblock PhD thesis, 1969.
\newblock Thesis (Ph.D.)--Brown University.

\bibitem{DPZ}
G.~Da~Prato and J.~Zabczyk.
\newblock {\em Stochastic equations in infinite dimensions}, volume~44 of {\em
  Encyclopedia of Mathematics and its Applications}.
\newblock Cambridge University Press, Cambridge, 1992.

\bibitem{DPZ2}
G.~Da~Prato and J.~Zabczyk.
\newblock {\em Ergodicity for Infinite-Dimensional Systems}, volume 229 of {\em
  LMS Lecture Note Series}.
\newblock London Mathematical Society, 1996.

\bibitem{Dal}
Yu.L. Daletski{\u\i}.
\newblock Stochastic differential geometry.
\newblock {\em Akad. Nauk Ukrain. SSR Inst. Mat. Preprint}, (46):48, 1982.

\bibitem{DHP}
R.~Denk, M.~Hieber, and J.~Pr{\"u}ss.
\newblock {$R$}-boundedness, {F}ourier multipliers and problems of elliptic and
  parabolic type.
\newblock {\em Mem. Amer. Math. Soc.}, 166(788), 2003.

\bibitem{DesLon}
G.~Desch and S.-O. Londen.
\newblock Maximal regularity for stochastic integral equations.
\newblock {\em Journal of Applied Analysis}, 19(1):125--140, 2013.

\bibitem{Dett91}
E.~Dettweiler.
\newblock Stochastic integration relative to {B}rownian motion on a general
  {B}anach space.
\newblock {\em Doga Mat}, 15(2):58--97, 1991.

\bibitem{GirRus2}
C.~Di~Girolami, G.~Fabbri, and F.~Russo.
\newblock The covariation for {B}anach space valued processes and applications.
\newblock {\em Metrica (Online First)}.
\newblock arXiv:1301.5715.

\bibitem{GirRus}
C.~Di~Girolami and F.~Russo.
\newblock Generalized covariation for {B}anach valued processes and {I}t{\^o}
  formula.
\newblock {\em Osaka J. Math.}
\newblock To appear, HAL-INRIA Preprint, 2010.

\bibitem{DJT}
J.~Diestel, H.~Jarchow, and A.~Tonge.
\newblock {\em Absolutely summing operators}, volume~43 of {\em Cambridge
  Studies in Advanced Mathematics}.
\newblock Cambridge University Press, Cambridge, 1995.

\bibitem{Dirksen}
S.~Dirksen.
\newblock {It\^o isomorphisms for {$L^p$}-valued {P}oisson stochastic
  integrals}.
\newblock {\em Ann. Probab.}
\newblock To appear, arXiv:1208.3885.

\bibitem{DMN}
S.~Dirksen, J.~Maas, and J.M.A.M.~van Neerven.
\newblock Poisson stochastic integration in {B}anach spaces.
\newblock {\em Electron. J. Probab.}, 18(100):1--28, 2013.

\bibitem{Gar}
D.J.H. Garling.
\newblock Brownian motion and {UMD}-spaces.
\newblock In {\em Probability and Banach spaces (Zaragoza, 1985)}, volume 1221
  of {\em Lecture Notes in Math.}, pages 36--49. Springer, Berlin, 1986.

\bibitem{Ga2}
D.J.H. Garling.
\newblock Random martingale transform inequalities.
\newblock In {\em Probability in Banach spaces 6 (Sandbjerg, 1986)}, volume~20
  of {\em Progr. Probab.}, pages 101--119. Birkh\"auser Boston, Boston, MA,
  1990.

\bibitem{Ge97}
S.~Geiss.
\newblock {${\rm BMO}\sb \psi$}-spaces and applications to extrapolation
  theory.
\newblock {\em Studia Math.}, 122(3):235--274, 1997.

\bibitem{GirWe03}
M.~Girardi and L.~Weis.
\newblock Operator-valued {F}ourier multiplier theorems on {$L_p(X)$} and
  geometry of {B}anach spaces.
\newblock {\em J. Funct. Anal.}, 204(2):320--354, 2003.

\bibitem{Haase:2}
M.H.A. Haase.
\newblock {\em The functional calculus for sectorial operators}, volume 169 of
  {\em Operator Theory: Advances and Applications}.
\newblock Birkh\"auser Verlag, Basel, 2006.

\bibitem{HHN02}
R.~Haller, H.~Heck, and A.~Noll.
\newblock Mikhlin's theorem for operator-valued {F}ourier multipliers in {$n$}
  variables.
\newblock {\em Math. Nachr.}, 244:110--130, 2002.

\bibitem{Hau11}
E.~Hausenblas.
\newblock Maximal inequalities of the {I}t\^o integral with respect to
  {P}oisson random measures or {L}\'evy processes on {B}anach spaces.
\newblock {\em Potential Anal.}, 35(3):223--251, 2011.

\bibitem{HJP}
J.~Hoffmann-J{\o}rgensen and G.~Pisier.
\newblock The law of large numbers and the central limit theorem in {B}anach
  spaces.
\newblock {\em Ann. Probability}, 4(4):587--599, 1976.

\bibitem{Hyt06}
T.P. Hyt{\"o}nen.
\newblock An operator-valued {$Tb$} theorem.
\newblock {\em J. Funct. Anal.}, 234(2):420--463, 2006.

\bibitem{KNVW}
N.~Kalton, J.M.A.M.~van Neerven, M.C. Veraar, and L.W. Weis.
\newblock Embedding vector-valued {B}esov spaces into spaces of
  {$\gamma$}-radonifying operators.
\newblock {\em Math. Nachr.}, 281(2):238--252, 2008.

\bibitem{KWcalc}
N.J. Kalton and L.W. Weis.
\newblock The {$H^\infty$}-calculus and sums of closed operators.
\newblock {\em Math. Ann.}, 321(2):319--345, 2001.

\bibitem{KaWe}
N.J. Kalton and L.W. Weis.
\newblock The {$H^\infty$}-calculus and square function estimates.
\newblock Preprint, 2004.

\bibitem{Kim}
K.-H. Kim.
\newblock Sobolev space theory of {SPDE}s with continuous or measurable leading
  coefficients.
\newblock {\em Stochastic Process. Appl.}, 119(1):16--44, 2009.

\bibitem{Kry94}
N.V. Krylov.
\newblock A generalization of the {L}ittlewood-{P}aley inequality and some
  other results related to stochastic partial differential equations.
\newblock {\em Ulam Quart}, 2(4):16, 1994.

\bibitem{Kry}
N.V. Krylov.
\newblock An analytic approach to {SPDE}s.
\newblock In {\em Stochastic partial differential equations: six perspectives},
  volume~64 of {\em Math. Surveys Monogr.}, pages 185--242. Amer. Math. Soc.,
  Providence, RI, 1999.

\bibitem{Kry00}
N.V. Krylov.
\newblock S{PDE}s in {$L\sb q((0,\tau]\!],L\sb p)$} spaces.
\newblock {\em Electron. J. Probab.}, 5(13):1--29, 2000.

\bibitem{Kry06}
N.V. Krylov.
\newblock On the foundation of the {$L\sb p$}-theory of stochastic partial
  differential equations.
\newblock In {\em Stochastic partial differential equations and
  applications---{VII}}, volume 245 of {\em Lect. Notes Pure Appl. Math.},
  pages 179--191. Chapman \& Hall/CRC, Boca Raton, FL, 2006.

\bibitem{Ku}
H.~Kunita.
\newblock Stochastic integrals based on martingales taking values in {H}ilbert
  space.
\newblock {\em Nagoya Math. J.}, 38:41--52, 1970.

\bibitem{KuWe}
P.C. Kunstmann and L.W. Weis.
\newblock Maximal {$L\sb p$}-regularity for parabolic equations, {F}ourier
  multiplier theorems and {$H\sp \infty$}-functional calculus.
\newblock In {\em Functional analytic methods for evolution equations}, volume
  1855 of {\em Lecture Notes in Math.}, pages 65--311. Springer, Berlin, 2004.

\bibitem{Kunze12}
M.C. Kunze.
\newblock Stochastic reaction-diffusion systems with {H}\"older continuous
  multiplicative noise.
\newblock arXiv:1209.4821, 2012.

\bibitem{KunzeMart}
M.C. Kunze.
\newblock On a class of martingale problems on {B}anach spaces.
\newblock {\em Electron. J. Probab.}, 18(104):1--30, 2013.

\bibitem{KuNe12}
M.C. Kunze and J.M.A.M. van Neerven.
\newblock Continuous dependence on the coefficients and global existence for
  stochastic reaction diffusion equations.
\newblock {\em J. Differential Equations}, 253(3):1036--1068, 2012.

\bibitem{KwWo89}
S.~Kwapie{\'n} and W.A. Woyczy{\'n}ski.
\newblock Tangent sequences of random variables: basic inequalities and their
  applications.
\newblock In {\em Almost everywhere convergence (Columbus, OH, 1988)}, pages
  237--265. Academic Press, Boston, MA, 1989.

\bibitem{LatOle99}
R.~Lata{\l}a and K.~Oleszkiewicz.
\newblock Gaussian measures of dilatations of convex symmetric sets.
\newblock {\em Ann. Probab.}, 27(4):1922--1938, 1999.

\bibitem{LeTa}
M.~Ledoux and M.~Talagrand.
\newblock {\em Probability in {B}anach spaces: isoperimetry and processes},
  volume~23 of {\em Ergebnisse der Mathematik und ihrer Grenzgebiete}.
\newblock Springer-Verlag, Berlin, 1991.

\bibitem{Lenglart}
E.~Lenglart.
\newblock Relation de domination entre deux processus.
\newblock {\em Ann. Inst. H. Poincar\'e Sect. B (N.S.)}, 13(2):171--179, 1977.

\bibitem{Maa10}
J.~Maas.
\newblock Malliavin calculus and decoupling inequalities in {B}anach spaces.
\newblock {\em J. Math. Anal. Appl.}, 363(2):383--398, 2010.

\bibitem{MN08}
J.~Maas and J.M.A.M.~van Neerven.
\newblock A {C}lark-{O}cone formula in {UMD} {B}anach spaces.
\newblock {\em Electron. Commun. Probab.}, 13:151--164, 2008.

\bibitem{marinelli2013maximal}
C.~Marinelli.
\newblock On maximal inequalities for purely discontinuous {$ L_q $}-valued
  martingales.
\newblock {\em arXiv:1311.7120}, 2013.

\bibitem{MPR}
C.~Marinelli, C.~Pr{\'e}v{\^o}t, and M.~R{\"o}ckner.
\newblock Regular dependence on initial data for stochastic evolution equations
  with multiplicative {P}oisson noise.
\newblock {\em Journal of Functional Analysis}, 258(2):616--649, 2010.

\bibitem{MR}
C.~Marinelli and M.~R{\"o}ckner.
\newblock Well-posedness and asymptotic behavior for stochastic
  reaction-diffusion equations with multiplicative {P}oisson noise.
\newblock {\em Electron. J. Probab}, 15:1528--1555, 2010.

\bibitem{Mau}
B.~Maurey.
\newblock Syst\`eme de {H}aar.
\newblock In {\em S\'eminaire Maurey-Schwartz 1974--1975: Espaces $L^{p}$,
  applications radonifiantes et g\'eom\'etrie des espaces de Banach, Exp. Nos.
  I et II}, page 26 pp. Centre Math., \'Ecole Polytech., Paris, 1975.

\bibitem{McCon84}
T.R. McConnell.
\newblock On {F}ourier multiplier transformations of {B}anach-valued functions.
\newblock {\em Trans. Amer. Math. Soc.}, 285(2):739--757, 1984.

\bibitem{McCon89}
T.R. McConnell.
\newblock Decoupling and stochastic integration in {UMD} {B}anach spaces.
\newblock {\em Probab. Math. Statist.}, 10(2):283--295, 1989.

\bibitem{MC}
T.R. McConnell.
\newblock Decoupling and stochastic integration in {UMD} {B}anach spaces.
\newblock {\em Probab. Math. Statist.}, 10(2):283--295, 1989.

\bibitem{Met1}
M.~M{\'e}tivier.
\newblock {\em Reelle und vektorwertige {Q}uasimartingale und die {T}heorie der
  stochastischen {I}ntegration}.
\newblock Lecture Notes in Mathematics, Vol. 607. Springer-Verlag, Berlin,
  1977.

\bibitem{Met2}
M.~M{\'e}tivier.
\newblock {\em Stochastic partial differential equations in
  infinite-dimensional spaces}.
\newblock Scuola Normale Superiore di Pisa, Quaderni. Scuola Normale Superiore,
  Pisa, 1988.
\newblock With a preface by G. Da Prato.

\bibitem{MetPel}
M.~M{\'e}tivier and J.~Pellaumail.
\newblock {\em Stochastic integration}.
\newblock Academic Press, New York, 1980.
\newblock Probability and Mathematical Statistics.

\bibitem{NeeCMA}
J.M.A.M.~van Neerven.
\newblock $\gamma$-{R}adonifying operators--a survey.
\newblock In {\em Spectral Theory and Harmonic Analysis (Canberra, 2009)},
  volume~44 of {\em Proc. Centre Math. Anal. Austral. Nat. Univ.}, pages 1--62.
  Austral. Nat. Univ., Canberra, 2010.

\bibitem{NVWco}
J.M.A.M.~van Neerven, M.C. Veraar, and L.W. Weis.
\newblock Conditions for stochastic integrability in {UMD} {B}anach spaces.
\newblock In {\em Banach spaces and their applications in analysis (in honor of
  Nigel Kalton's 60th birthday)}, pages 127--146. De Gruyter Proceedings in
  Mathematics, De Gruyter, 2007.

\bibitem{NVW1}
J.M.A.M.~van Neerven, M.C. Veraar, and L.W. Weis.
\newblock Stochastic integration in {UMD} {B}anach spaces.
\newblock {\em Ann. Probab.}, 35(4):1438--1478, 2007.

\bibitem{NVW3}
J.M.A.M.~van Neerven, M.C. Veraar, and L.W. Weis.
\newblock Stochastic evolution equations in {UMD} {B}anach spaces.
\newblock {\em J. Funct. Anal.}, 255(4):940--993, 2008.

\bibitem{NVW-gamma}
J.M.A.M.~van Neerven, M.C. Veraar, and L.W. Weis.
\newblock Maximal gamma-regularity.
\newblock {\em arXiv preprint arXiv:1209.3782}, 2012.

\bibitem{NVW12eq}
J.M.A.M.~van Neerven, M.C. Veraar, and L.W. Weis.
\newblock {Maximal $L^p$-regularity for stochastic evolution equations.}
\newblock {\em SIAM J. Math. Anal.}, 44(3):1372--1414, 2012.

\bibitem{NVW10}
J.M.A.M.~van Neerven, M.C. Veraar, and L.W. Weis.
\newblock Stochastic maximal {$L^p$}-regularity.
\newblock {\em Ann. Probab.}, 40(2):788--812, 2012.

\bibitem{NVW10b}
J.M.A.M.~van Neerven, M.C. Veraar, and L.W. Weis.
\newblock {On the $R$-boundedness of stochastic convolution operators}.
\newblock Submitted, arXiv:1404.3353, 2014.

\bibitem{NW1}
J.M.A.M.~van Neerven and L.W. Weis.
\newblock Stochastic integration of functions with values in a {B}anach space.
\newblock {\em Studia Math.}, 166(2):131--170, 2005.

\bibitem{vNWe}
J.M.A.M.~van Neerven and L.W. Weis.
\newblock Weak limits and integrals of {G}aussian covariances in {B}anach
  spaces.
\newblock {\em Probab. Math. Statist.}, 25(1, Acta Univ. Wratislav. No.
  2784):55--74, 2005.

\bibitem{Nu}
D.~Nualart.
\newblock {\em Malliavin calculus and its applications}, volume 110 of {\em
  CBMS Regional Conference Series in Mathematics}.
\newblock Published for the Conference Board of the Mathematical Sciences,
  Washington, DC, 2009.

\bibitem{Ondrejat04}
M.~Ondrej{\'a}t.
\newblock Uniqueness for stochastic evolution equations in {B}anach spaces.
\newblock {\em Dissertationes Math. (Rozprawy Mat.)}, 426:63, 2004.

\bibitem{OndSei}
M.~Ondrej{\'a}t and J.~Seidler.
\newblock On existence of progressively measurable modifications.
\newblock {\em Electron. Commun. Probab.}, 18:1--6, 2013.

\bibitem{OndrVer2013}
M.~Ondrej\'at and M.C. Veraar.
\newblock Weak characterizations of stochastic integrability and {D}udley's
  theorem in infinite dimensions.
\newblock {\em J. Theoretical Probab.}, 2013.
\newblock Online First.

\bibitem{Pin}
I.~Pinelis.
\newblock Optimum bounds for the distributions of martingales in {B}anach
  spaces.
\newblock {\em The Annals of Probability}, 22(4):1679--1706, 1994.
\newblock Correction: {\em ibid.} 27:2119, 1999.

\bibitem{Pi75}
G.~Pisier.
\newblock Martingales with values in uniformly convex spaces.
\newblock {\em Israel J. Math.}, 20(3-4):326--350, 1975.

\bibitem{Pis88}
G.~Pisier.
\newblock Riesz transforms: a simpler analytic proof of {P.A.} {M}eyer's
  inequality.
\newblock {\em S{\'e}minaire de probabilit{\'e}s XXII}, pages 485--501, 1988.

\bibitem{PisConv}
G.~Pisier.
\newblock {\em The volume of convex bodies and {B}anach space geometry},
  volume~94 of {\em Cambridge Tracts in Mathematics}.
\newblock Cambridge University Press, Cambridge, 1989.

\bibitem{ProVer}
M.~Pronk and M.C Veraar.
\newblock Tools for {M}alliavin calculus in {UMD} {B}anach spaces.
\newblock {\em Potential Analysis}, 40:307--344, 2014.

\bibitem{RY}
D.~Revuz and M.~Yor.
\newblock {\em Continuous martingales and {B}rownian motion}, volume 293 of
  {\em Grundlehren der Mathematischen Wissenschaften}.
\newblock Springer-Verlag, Berlin, 1991.

\bibitem{RS}
J.~Rosi{\'n}ski and Z.~Suchanecki.
\newblock On the space of vector-valued functions integrable with respect to
  the white noise.
\newblock {\em Colloq. Math.}, 43(1):183--201 (1981), 1980.

\bibitem{SchnVer10}
R.~Schnaubelt and M.C. Veraar.
\newblock Structurally damped plate and wave equations with random point force
  in arbitrary space dimensions.
\newblock {\em Differential Integral Equations}, 23(9-10):957--988, 2010.

\bibitem{Seidler}
J.~Seidler.
\newblock Exponential estimates for stochastic convolutions in 2-smooth
  {B}anach spaces.
\newblock {\em Electron. J. Probab.}, 15:no. 50, 1556--1573, 2010.

\bibitem{StrWeis08}
{\v{Z}}.~{\v{S}}trkalj and L.~W. Weis.
\newblock On operator-valued {F}ourier multiplier theorems.
\newblock {\em Trans. Amer. Math. Soc.}, 359(8):3529--3547 (electronic), 2007.

\bibitem{Yo}
M.~Yor.
\newblock Sur les int\'egrales stochastiques \`a valeurs dans un espace de
  {B}anach.
\newblock {\em Ann. Inst. H. Poincar\'e Sect. B (N.S.)}, 10:31--36, 1974.

\bibitem{Zhu}
J.~Zhu.
\newblock Maximal inequalities for stochastic convolutions driven by {L}\'evy
  processes in {B}anach spaces.
\newblock Work in progress.

\bibitem{Zim89}
F.~Zimmermann.
\newblock On vector-valued {F}ourier multiplier theorems.
\newblock {\em Studia Math.}, 93(3):201--222, 1989.

\end{thebibliography}
\end{document}